\documentclass{article}
\usepackage[left=2.5cm,right=2.5cm,top=2cm,bottom=2cm,a4paper]{geometry}
\usepackage{amsthm}
\usepackage{amsmath}	\usepackage{amssymb}	\usepackage[UKenglish]{babel}		\usepackage[colorlinks=true]{hyperref}
\usepackage{cleveref}	\usepackage{csquotes}	\usepackage[explicit]{titlesec}
\usepackage{titletoc}
\usepackage{graphicx}	\usepackage{mathrsfs}	\usepackage{mathtools}
\usepackage{setspace}
\usepackage{slashed}
\usepackage{xcolor}
\usepackage{tikz}
\usepackage{microtype}	\usepackage[backend=biber]{biblatex}
\usepackage{svg}
\usepackage{fancyhdr}	\usepackage{enumitem}	\MakeOuterQuote{"}

\newcounter{countall}

\theoremstyle{plain}
\newtheorem{definition}[countall]{Definition}
\newtheorem{theorem}[countall]{Theorem}

\newtheorem{lemma}[countall]{Lemma}
\newtheorem{prop}[countall]{Proposition}

\newtheorem*{definition*}{Definition}
\newtheorem*{theorem*}{Theorem}
\newtheorem*{remark*}{Remark}
\newtheorem*{lemma*}{Lemma}
\newtheorem*{prop*}{Proposition}
\newtheorem*{cor*}{Corollary}
\newtheorem*{example*}{Example}
\newtheorem*{notation*}{Notation}

\newenvironment{proclaimenumerate}{\begin{enumerate}[label=(\roman*),nosep]}
{\end{enumerate}}

\def\colormacro#1#2{{\color{#1}#2}}

\def\cbrak#1{\left\{#1\right\}}
\def\rbrak#1{\left(#1\right)}
\def\sbrak#1{\left[#1\right]}

\def\floor#1{\left\lfloor#1\right\rfloor}

\def\suchthat{\medspace\text{s.t.}\medspace}
\def\stbar{\medspace|\medspace}

\def\ie{{\it i.e.}, }
\def\eg{{\it e.g.}, }

\def\etc{{\it etc.}}

\def\adhoc{{\it ad-hoc}}

\def\al{\alpha}
\def\be{\beta}
\def\Ga{\Gamma}
\def\ga{\gamma}

\def\eps{\epsilon}

\def\la{\lambda}

\def\Si{\Sigma}
\def\si{\sigma}

\def\om{\omega}

\def\Om{\Omega}

\def\C{\mathbb{C}}

\def\R{\mathbb{R}}

\def\calA{\mathcal{A}}
\def\calB{\mathcal{B}}
\def\calC{\mathcal{C}}

\def\calE{\mathcal{E}}

\def\calL{\mathcal{L}}

\def\calV{\mathcal{V}}
\def\calW{\mathcal{W}}

\def\dist{\mathop{\mathrm{dist}}\limits}

\def\coker{\mathop{\mathrm{coker}}}
\def\tr{\mathop{\mathrm{tr}}}

\def\iso{\cong}

\def\Area{\mathop{\mathrm{Area}}}

\def\dvg{\mathop{\mathrm{div}}}

\def\into{\hookrightarrow}

\def\norm#1{\left\| #1 \right\|}
\def\abs#1{\left| #1 \right|}
\def\inner#1{\left\langle #1 \right\rangle}
\def\tensor{\otimes}

\def\ed{\mathrm{d}}\def\zb{{\overline{z}}}
\def\dz{\mathrm{d}z}
\def\dzb{\mathrm{d}\overline{z}}
\def\pd{\partial}
\def\pdz{{\partial_z}}
\def\pdzb{{\partial_{\overline{z}}}}
\def\deriv#1#2{\frac{\mathrm{d}{#1}}{\mathrm{d}{#2}}}
\def\pderiv#1#2{\frac{\partial{#1}}{\partial{#2}}}

\def\derivat#1#2#3{\frac{\mathrm{d}{#1}}{\mathrm{d}{#2}}\Big|_{#3}}

\def\dbar{\overline{\partial}}

\def\firstvar#1#2#3{\mathrm{d}{#1}\!\rbrak{#2}\sbrak{#3}}
\def\secondvar#1#2#3{\mathrm{d}^2{#1}\!\rbrak{#2}\sbrak{#3}}

\def\gadot{\dot{\gamma}}
\def\Ric{\mathop{\mathrm{Ric}}}
\def\Rm{\mathop{\mathrm{Rm}}}
\def\Vol{\mathrm{Vol}}

\def\tgt{\top}
\def\sff{\mathrm{II}}

\def\geomint#1#2{\int_{#1}\!#2}
\def\integ#1#2#3#4{\int_{#1}^{#2}\!{#3}\,\mathrm{d}{#4}}

\def\indx{\mathop{\mathrm{index}}}

\def\smooth#1{C^\infty\!\left(#1\right)}

\def\exclude#1{{\colormacro{gray}{\texttt{excluded text}}\\}}

	\def\cmc{{\sc cmc}}

\author{Luca Seemungal}
\title{Index Estimates for CMC and Minimal Surfaces with Capillary Boundary}

\newcommand{\rmn}{\ensuremath{\mathrm{n}}}

\begin{document}
\maketitle

\begin{abstract}
We prove that the index of a \cmc{} surface with capillary boundary is bounded from above linearly by its genus, number of boundary components, and branching order, and also by some Willmore-type energy involving the area, mean curvature, contact angle, and ambient curvature.
The main auxiliary theorem of more general interest is a comparison of the second variations of area and energy at a branched conformal map with boundary.
In the appendix we derive the various second variation formulae for area, enclosed-volume, and wetting functionals away from critical points and for non-admissible variations, the purpose of which is to rather comprehensively fill a gap in the literature.
\end{abstract}

\section{Introduction}

Let $\Si$ be a compact Riemann surface with boundary, and let $M^3$ be a Riemannian three-manifold with boundary.
It is well-known that, for \emph{boundaried (branched) conformal maps}\footnote{We call a map $u$ \emph{boundaried} if $u(\pd\Si)\subset\pd M$. We call a map \emph{branched} if $\ed u$ vanishes on a discrete subset of the domain and these singularities have the structure of branch points, as in \cite{gulliver-osserman-royden-1973}.} $u:\Si\to M$ , the area $\calA(u)$ and Dirichlet energy $\calE(u)$ agree up to zero'th and first order:
$$ u \text{ conformal} \implies \calA(u)=\calE(u) \text{ and } \ed\calA(u)\equiv\ed\calE(u). $$
One is therefore behoven to study the second order relationship between them.

Indeed, the idea of comparing the area and energy goes at least as far back as the work of J.~Douglas (see, for instance, \cite[Sec.~5.3]{douglas-1939}), but it was not until early in this millennium that N.~Ejiri and M.~Micallef performed the first detailed comparison of area and energy \cite{ejiri-micallef-2008}, where they computed the second order deficit $\secondvar{\rbrak{\calE-\calA}}{u}{v,v}$ for a \emph{minimal surface} $u$ (a conformal map which is also harmonic), in terms of a certain $\dbar$-problem dependent on $v$.

That such a comparison might be achievable is suggested by the following heuristic, firstly and elegantly exposited by Ejiri and Micallef, and only a brief outline of which we give here.
As is well known, the area is invariant the group of diffeomorphisms of the domain, but the energy is only invariant under the subgroup of conformal diffeomorphisms.
Starting therefore at a conformal map, one expects on the one hand that variations in the target that induce conformal diffeomorphisms will keep area and energy the same, whilst on the other hand that variations which do not induce conformal diffeomorphisms will incur some kind of \emph{conformal deficit} between the area and the energy.
Of course, that some variations \emph{cannot} induce conformal diffeomorphisms is revealed by the dimension of the moduli space of conformal structures on the domain.
Ejiri and Micallef formalised this heuristic and used the comparison to  prove index bounds for branched minimal surfaces.

For minimal surfaces, the second variation of area is given by $q=\abs{A}^2+\Ric(\nu,\nu)$, where $A$ is the second fundamental form, $\Ric$ is the ambient Ricci curvature, and $\nu$ is the normal to the minimal surface.
Unfortunately, the fact is that, given $L=-\Delta-q$ a Schrödinger operator on a two-dimensional domain\footnote{This is in contrast to higher dimensional domains, where one does have the Cwikel--Lieb--Rosenbljum inequality, see for instance \cite{li-yau-1983}.} $\Om$, it is \emph{not} possible (as simple examples of $q$ show) to have a general inequality of the form (here $i_L$ is the number of negative eigenvalues of $L$)
$$ i_L \leq C\geomint{\Om}{\abs{q_+}}; $$
such an inequality would give a bound of the desirous form $i_{\calA}\leq C(M)\calA(u) + Cg$, where $C(M)$ depends on $M$ and $g$ is the genus of the surface.
Fortuitously, there are by now established methods to bound the index of the energy.
For the sake of exposition, we infidelitiously reduce the work of Ejiri--Micallef down to two steps.
Firstly, the area--energy comparison $\secondvar{\rbrak{\calE-\calA}}{u}{v,v}$ gives the comparison of area and energy indices $i_{\calE} \leq i_{\calA} \leq i_{\calE}+r$, where $r$ (simply put) is the dimension of the Teichmüller space, and therefore a topological-branching term.
Secondly, one bounds $i_{\cal{E}}$ by relevant geometric quantities (area in the compact case) using heat-kernel methods \cite{cheng-tysk-1994}.
The upshot is that one gets the desired index bounds of the form (for, say, compact minimal surfaces) $i_{\calA} \leq C\calA(u) + r$, where $C$ is a constant depending only on the ambient space $M$ and $r$ is some topological term.
Let us also mention that V.~Lima has carried out this procedure for free-boundary minimal surfaces \cite{lima-2022}, computing firstly the second variation comparison $\secondvar{\rbrak{\calE-\calA}}{u}{v,v}$ at a free-boundary minimal surface, and secondly applying these heat-kernel methods to deduce the bounds.

Most of that which I have hitherto described has assumed the surface to be minimal (which, one recalls, is a conformal and harmonic map).
A brief glance at the heuristic however reveals that it should apply to arbitrary conformal maps.
Indeed, together with B.~Sharp, we computed \cite{seemungal-sharp-2026} the second variation comparison $\secondvar{\rbrak{\calE-\calA}}{u}{v,v}$ at an arbitrary conformal map, but without boundary.
Yet another brief glance reveals that it should also apply to arbitrary conformal maps with boundary, which is our main auxiliary theorem below.
The proof we give is different to that in \cite{seemungal-sharp-2026}, in addition to dealing with multitude of terms appearing on the boundary.
\begin{theorem*}[Theorem \ref{thm:area-energy-comparison} in main body]
Let $M$ be a Riemannian three-manifold (possibly with boundary), let $\Si$ be a Riemann surface (possibly with boundary), and suppose that $u:\Si\to M$ is a boundaried branched conformal immersion.
Then, for any variation $\Phi:I\times\Si\to M$ of $u$ such that $\Phi(I,\pd\Si)\subset\pd M$, we have
\begin{equation*}
\derivat{^2}{t^2}{t=0}{\rbrak{\calE - \calA}\rbrak{\Phi(t)}} = 4\int_\Si\!\abs{\mu}^2\,\ed\Si,
\end{equation*}
where, denoting $v:=\Phi'(0)\in\Ga_{\pd M}(u^*TM)$ and splitting $v=s+\si$ into $s:=v^\perp$ and $\si:=v^\tgt$,
$$ \mu = \rbrak{
	\nabla^\tgt_\pdz \si^{\rbrak{0,1}}
	- 2e^{-2\la}\inner{s,A(u_z,u_z)}u_\zb
}\tensor\dz $$
is the \emph{infinitesimal conformal deficit} of $v$.
\end{theorem*}
\begin{remark*}
\begin{proclaimenumerate}
\item For an elucidation of the so-called "infinitesimal conformal deficit" $\mu$, see Section \ref{sec:comparison}.
Suffice it to say for now that $\mu$ should be viewed as a PDE which asks, given a normal variation $s$, to find a tangential variation $\si$ which reparametrises the domain so as to stay conformal (at the infinitesimal level, of course).
\item What we have written above is a comparison of the \emph{two derivatives} of $\calE$ and $\calA$, which (since $u$ is not necessarily a critical point of $\calE$ nor $\calA$) is not the same as the comparison of the ``second variation'' or \emph{Hessian} of $\calE$ and $\calA$.
Indeed, it is not clear from the outset that the left-hand side is even bilinear.
Nevertheless, the right-hand side is bilinear, so the non-bilinearities in the second derivatives must cancel precisely (as one can see in the proof).
For this reason, if one wished to write the comparison as one of the Hessians of $\calE$ and $\calA$ over some infinite-dimensional space of variations, then one could readily deduce this from what we have here, even though this is not necessary in this paper.
\item Weaker forms of this area--energy comparison for minimal surfaces has appeared in various works, (for instance, \cite{cheng-zhou-2023,gao-zhu-2024}), where the computations are much simplified by somewhat weaker theorem statements (\ie{}the area--energy comparison is an inequality rather than an equality) and with assumptions about the target or the domain being particularly special (\eg{}a sphere), and assumptions on the map $u$ (being harmonic).
\end{proclaimenumerate}
\end{remark*}

As alluded to above, the second variation comparison gives rise (after some argumentation) to the comparison of area and energy indices, Theorem \ref{thm:index-area-energy-comparison}.
Adapting the various techniques that appear in the literature, we bound the energy index in Theorem \ref{thm:energy-index-bound}.
Combining these two theorems gives our main theorem below.
We adopt the term \emph{$h$-\cmc{} surface with $\theta$-capillary boundary} for a \cmc{} surface with mean curvature $h>0$ and contact angle $\theta\in(0,\pi/2)$; see Section \ref{sec:prelims} for details.

\begin{theorem}
\label{thm:main}
There exists a constant $C>0$ such that for any Riemannian manifold $M$ with boundary, isometrically embedded into some Euclidean space $\R^d$, and for any oriented Riemann surface $\Si$ of genus $g$ with $m\neq0$ boundary components, if $u:\Si\to M$ is an $h$-\cmc{} surface with $\theta$-capillary boundary in $\pd M$ then
$$ i_\Si \leq C\rbrak{1+\frac{1}{\sin\theta}}^2(J^2 + B^2 + h^2)\calA(u) + r $$
where $h$ is the mean curvature of $u$, $J$ is the sup-norm of the largest eigenvalue of the second fundamental form of the aforementioned isometric embedding $M\into\R^d$, $B$ is the sup-norm of the largest eigenvalue of the second fundamental form of $\pd M$ in $M$, $\calA(u)$ is the (mapping-)area of $u$, and
$$ r =
\begin{cases}
6g - 6 + 3m - 2b - d &\text{if }2b+d<4g-4+2m,\\
4g - 2 + 2m - 2b + 2\floor{-d/2} &\text{if } 4g-4+2m\leq2b+d\leq8g-8+4m,\\
0 &\text{if }8g-8+4m<2b+d,
\end{cases}
$$
where $b$ is the interior branching order and $d$ is the boundary branching order.
\end{theorem}

An easy modification (in fact, a simplification) of the arguments presented in this article yield the corresponding index bounds on branched minimal surfaces with capillary boundary.
These are novel, generalising previous results of Lima \cite{lima-2022} which apply to unbranched minimal surfaces with standard free boundary contact angle $\theta=\pi/2$.

\paragraph{Acknowledgements.}
The author is greatly indebted to B. Sharp for countless illuminating discussions regarding this work, and to M. Micallef for his interest, encouragement, and advice.
This work has been supported by the \textsc{epsrc} grant \texttt{EP/W523860/1}, project reference \texttt{2758306}.

\section{Preliminaries}
\label{sec:prelims}

In this section we record our notations, conventions, and definitions.
Throughout, unless otherwise stated, $\Si$ will be a compact Riemann surface-with-boundary of genus $g$, and $M$ will be a compact Riemannian three-manifold with boundary.
We call a map $u:\Si\to M$ \emph{boundaried} if $u(\pd\Si)\subset\pd M$.

Let $u:\Si\to M$ be a boundaried branched conformal immersion.
In accordance with the Riemannian metric then, the pull-back tangent bundle splits $u^*TM=\xi\oplus\nu\Si$, where $\xi$ is the ramified tangent bundle, and $\nu\Si$ is the normal bundle to $\Si$.
The ramified tangent bundle $\xi$ is the usual tangent bundle of $\Si$ which is "twisted at the branch points"; this phenomenon is perhaps best illustrated using complex parameters.
Indeed after complexifying $\xi$, we see that $\xi^{1,0}=T^{1,0}\Si\tensor[D]$ where $D$ is the divisor corresponding to the branch points and $[D]$ is the associated line bundle.
The natural correspondence then (by the complex parameter) between the real tangent bundle $\xi$ and the holomorphic tangent bundle $\xi^{1,0}$ elucidates precisely what is meant by "twisting the tangent bundle".
Indeed, we shall be making local computations in a complex parameter $z=x+iy$, where the pull-back metric takes the form $u^*g=e^{2\la}\rbrak{\ed x^2+\ed y^2}$, along with the standard notations $\pd_z=\frac{1}{2}\rbrak{\pd_x-i\pd_y}$ \etc{}
We denote by $\nabla$ both the Levi-Civita connection on $TM$ and the pull-back by $u$ of the same; moreover, $\tgt$ denotes tangential projection onto $\Si$, $\perp$ denotes projection normal to $\Si$, and $A=\rbrak{\nabla_\cdot\cdot}^\perp$ denotes the second fundamental form of $u:\Si\to M$.

At the boundary of our immersions, one has the various normals, conormals, \etc, whose definition depends on various choices of orientation; we set this all here in the following way.
Perhaps it is clearer to refer to \cref{fig:normals}.
\begin{center}
\begin{tabular}{p{0.05\linewidth} p{0.8\linewidth}}
$\dot{\ga}$	& The unit tangent defined, in a local chart mapping the upper-half-plane to $\Si$ and the real axis to $\pd\Si$, by $\dot{\ga}=e^{-\la}\pd_x$. \\
$\rmn$		& The outward-pointing unit conormal to $\pd\Si$, defined locally by $\rmn=-e^{-\la}\pd_y$ when the local chart is as above. \\
$\nu$		& The unit normal to $\Si$ defined by $\nu:=-\star\dot{\ga}\wedge\rmn$, so that $\ed\Si\wedge\nu=\Vol^M$. \\
$N$			& The outward-pointing unit normal to $\pd M$. \\
$\hat{\nu}$	& The outward-pointing unit normal to $\pd\Si$ in $\pd M$ defined by the relation $\dot{\ga}\wedge\hat{\nu}\wedge N=\star 1$. \\
\end{tabular}
\end{center}
Note that at a boundary branch point, $\gadot$ and $\rmn$ vanish; nevertheless, the unit normal $\nu$ is well defined.
We denote by $A^{\pd M}=\inner{\nabla_\cdot\cdot,N}$ the second fundamental form of $\pd M$ in $M$.

\begin{figure}
\begin{center}
\begin{tikzpicture}[scale=1.8]
\fill [fill=black,opacity=0.2] (-1.5,0) rectangle (1.5,1.5);
\node [anchor=east] at (1.25,1.25) {$M$};
\draw [thick] (-1.5,0) -- (-0.9,0) node [below] {$\pd M$} -- (1.5,0);
\draw (-1.5,1.5) to [out=-90,in=150]
	node [sloped,above] {$\Si$} (0,0);
\fill (0,0) circle (0.05cm);
\draw[very thick,->](0,0) -- (0,-1)
	node [below] {$N$};
\draw[very thick,->](0,0) -- (canvas polar cs:angle=-30,radius=1cm)
	node [anchor=north west] {$\rmn$};
\draw[very thick,->](0,0) -- (1,0)
	node [anchor=south west] {$\hat{\nu}$};
\draw[very thick,->](0,0) -- (canvas polar cs:angle=60,radius=1cm)
	node [anchor=south west] {$\nu$};
\draw (60:0.2) -- ++(150:0.2) -- ++(240:0.2);
\draw (0,-0.2) -- (0.2,-0.2) -- (0.2,0);
\fill [fill=black,opacity=0.1] (0,0) -- +(0:0.7cm)
	arc [start angle=0,end angle=-30,radius=0.7cm] -- cycle;
\draw [draw=black,opacity=1] (0,0) +(0:0.7cm)
	arc [start angle=0,end angle=-30,radius=0.7cm];
\path (0,0) ++ (-15:0.5cm) node {$\alpha$};
\end{tikzpicture}
\end{center}
\caption{The conventions for the normals, conormals, \etc\ Note that the unit tangent $\dot{\ga}$ to $\pd\Si$ points \textit{away} from the reader (in accordance with the "left-hand rule"), and does not appear in the figure.}
\label{fig:normals}
\end{figure}
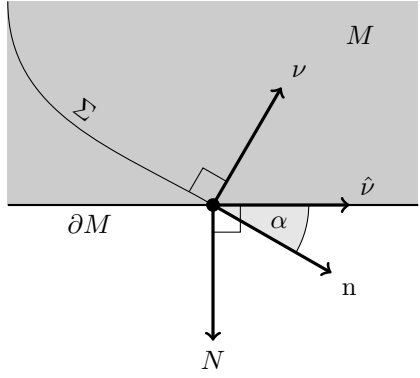

We have the usual area functional, being
\begin{equation}\label{eqn:area-defn}
\calA(u) :=  \geomint{\Si}{u^*\Vol^M.}
\end{equation}
for boundaried branched immersions $u:\Si\to M$.

The volume and wetting functionals are slightly different, in that they are defined in an \adhoc{} manner.
Indeed, if $\Phi(t):(-\eps,\eps)\times\Si\to M$ is a variation of $u=\Phi(0)$ with variational vector field $v:=\Phi'(0)$, we can define the (weighted) enclosed-volume functional as
\begin{equation}\label{eqn:volume-defn}
\calV^h(\Phi(t)) = \geomint{[0,t]\times\Si}{\Phi^*\rbrak{h\Vol^M}},
\end{equation}
where we take the putative prescribed mean curvature $h>0$ to be constant.\footnote{Much of this, as well as for the contact angle below, works for non-constant prescribed mean curvature and contact angle (say, $h\in\smooth{M}$ and $\theta\in\smooth{\pd M}$).
Indeed, in the appendix, we leave many of the formulae in general form, not assuming constancy.
We do not however dwell on that here in the main body of the paper since we deem this would be a distraction somewhat.}
We have the first variation formula
\begin{equation*}
\firstvar{\calV^h}{u}{v}
=
\geomint{\Si}{\inner{v,h\nu}\,\ed\Si},
\end{equation*}
from which we see that $\ed\calV^h$ is a well-defined functional on the space of boundaried branched immersions, in the sense that $\ed\calV^h$ depends only on $u$ and $v$, not on the whole variation $\Phi$.

Similarly, for a variation as above, we can define the (weighted) wetting functional to be
\begin{equation}\label{eqn:wetting-defn}
\calW^\theta(\Phi(t)) = \geomint{[0,t]\times\pd\Si}{\Phi^*\rbrak{\cos\theta\Area^{\pd M}}},
\end{equation}
where $\theta\in(0,\pi/2)$ is the putative prescribed contact angle.
Once again, the first variation formula
\begin{equation*}
\firstvar{\calW^\theta}{u}{v} = -\geomint{\pd\Si}{\cos\theta\inner{v,\hat{\nu}}\,\ed\tau}
\end{equation*}
(where $\tau$ is an arc-length parameter of $\pd\Si$) shows that $\ed\calW^\theta$ is a well-defined functional on the space of boundaried branched immersions.

We may now define the critical object we deal with in this paper, the \emph{capillary \cmc{} surface}.
Such surfaces are critical points of the area functional with respect to variations preserving enclosed-volume and boundary-area.
We write $\calV$ and $\calW$ for the unweighted functionals.

\begin{definition}\label{defn:capillary}
Let $u:\Si\to M$ be a boundaried branched conformal immersion.
An \emph{admissible variation} of $u$ is a smooth family of maps $\Phi:(-\eps,\eps)\times\Si\to M$ with $\Phi(0)=u$ such that firstly
enclosed-volume is preserved,
\begin{equation*}
\calV(\Phi(t)) = \calV(u) \textrm{ for all }t\in(-\eps,\eps),
\end{equation*}
and secondly boundary-area is preserved,
\begin{equation*}
\calW(\Phi(t)) = \calW(u) \textrm{ for all }t\in(-\eps,\eps).
\end{equation*}
The map $u$ is a \emph{capillary \cmc{} surface} if for all admissible variations $\Phi$,
$$ \derivat{}{t}{t=0}\calA(\Phi(t)) = 0. $$
\end{definition}

One also has the notion of the \textit{infinitesimal admissible variation} of $u$, which is any vector field $v\in\Ga(u^*TM)$ satisfying $v|_{\pd\Si}\in\Ga(T\pd M)$, $\firstvar{\calV}{u}{v}=0$, and $\firstvar{\calW}{u}{v}=0$.
An admissible variation $\Phi$ easily gives rise to a corresponding infinitesimal admissible variation $v:=\Phi'(0)$, and it is a lemma of \cite{barbosa-do-carmo-eschenburg-1988} that given an infinitesimal admissible variation $v$, there exists an admissible variation $\Phi$ with $\Phi'(0)=v$.
Note also that, if $u$ is conformal (which we can always assume), then $\ed\calA(u)\equiv\ed\calE(u)$.
We therefore have the following equivalent definitions.

\begin{prop}
Let $u:\Si\to M$ be a boundaried branched conformal immersion, and denote
$$
\Ga_{\pd M}(u^*TM) := \cbrak{v\in\Ga(u^*TM)\stbar\text{ if }p\in\pd\Si\text{ then }v(p)\in T_{u(p)}\pd M}.
$$
Then the following are equivalent:
\begin{proclaimenumerate}
\item $u$ is a capillary \cmc{} surface;
\item $\firstvar{\calA}{u}{v}=0$ for all $v\in\Ga_{\pd M}(u^*TM)$ such that $\firstvar{\calV}{u}{v}=\firstvar{\calW}{u}{v}=0$;
\item there exist $h>0$ and $\theta\in(0,\pi/2)$ such that $\firstvar{\rbrak{\calA+\calV^h+\calW^\theta}}{u}{v}=0$ for all $v\in\Ga_{\pd M}(u^*TM)$;
\item $\firstvar{\calE}{u}{v}=0$ for all $v\in\Ga_{\pd M}(u^*TM)$ such that $\firstvar{\calV}{u}{v}=\firstvar{\calW}{u}{v}=0$;
\item there exist $h>0$ and $\theta\in(0,\pi/2)$ such that $\firstvar{\rbrak{\calE+\calV^h+\calW^\theta}}{u}{v}=0$ for all $v\in\Ga_{\pd M}(u^*TM)$.
\end{proclaimenumerate}
\end{prop}
\textbf{Henceforth (and without loss of generality) we shall assume capillary \cmc{} surfaces to be conformally parametrised.}
Naturally, it is a consequence of this proposition that a capillary \cmc{} surface has constant mean curvature $h$ and constant contact angle $\theta$ (see for instance Lemma \ref{lem:vars-combined-general}); we refer to such surfaces as \emph{h-\cmc{} $\theta$-capillary surfaces} (or something similar).
We shall be making great use of the modified area and energy functionals, for which reason we set the following notation here;
\begin{align*}
\calE^{h,\theta} &:= \calE + \calV^h + \calW^\theta \\
\calA^{h,\theta} &:= \calA + \calV^h + \calW^\theta.
\end{align*}

We are interested in the stability of \cmc{} surfaces.
For us, this means that we will estimate the number of negative eigenvalues of some quadratic forms.

\begin{definition}
\label{def:index}
Let $u:\Si\to M$ be a $\theta$-capillary $h$-\cmc{} surface.
The \emph{index} of $u$ is the number of negative eigenvalues of the quadratic form $Q$ given by, for $f\in\smooth{\Si,\R}$,
\begin{multline}
Q(f,f) =
\geomint{\Si}{-f\Delta f - \Ric^M(\nu,\nu)f^2 - \abs{A}^2f^2\,\ed\Si}
\\
+ \geomint{\pd\Si}{f\pderiv{f}{\rmn}
	+ \rbrak{
		\cot\theta\inner{A(\rmn,\rmn),\nu}
		+ \frac{1}{\sin\theta}\inner{A^{\pd M}(\hat{\nu},\hat{\nu}),N}
	}f^2\,\ed\tau.
}
\end{multline}
The \emph{energy index} of $u$ is the number of negative eigenvalues of the quadratic form $Q_{\calE}$ given by, for $v\in\Ga_{\pd M}(u^*TM)$,
\begin{multline}
Q_{\calE}(v,v) =
\geomint{\Si}{-\inner{v,\Delta v}
	- \Rm^M(v,E_i,E_i,v)
	+ h\Vol^M(v,\nabla_{u_x}v,u_y)
	+ h\Vol^M(v,u_x,\nabla_{u_y}v)
	\,\ed\Si}
\\
+ \geomint{\pd\Si}{
	\inner{v,\nabla_\rmn v}
	- \sin\theta\inner{A^{\pd M}(\dot{\ga},\dot{\ga}),N}\abs{v}^2
	\,\ed\tau
}.
\end{multline}
\end{definition}
These quadratic forms arise out of the Hessians of the functionals $\calA^{h,\theta}$ and $\calE^{h,\theta}$ derived in the appendix --- specifically, Lemmas \ref{lem:hess-mod-area-cmc-cap} and \ref{lem:hess-mod-energy}.
In this article, we estimate the number of negative eigenvalues of $Q$ over all variations, not just those which are admissible.
Nevertheless, it is readily seen that the error is at most one.

\section{Comparison of area and energy}
\label{sec:comparison}

It is an elementary and classical application of the Cauchy--Schwarz inequality that $\calA(u)\leq\calE(u)$ with equality if and only if $u$ is conformal.
Furthermore, if $u$ is conformal, then $\ed\calA(u)\equiv\ed\calE(u)$.
At a conformal map, the relationship between the second variations is given by the following.

\begin{theorem}
\label{thm:area-energy-comparison}
Let $M$ be a Riemannian three-manifold (possibly with boundary), let $\Si$ be a Riemann surface (possibly with boundary), and suppose that $u:\Si\to M$ is a boundaried branched conformal immersion.
Then, for any variation $\Phi:I\times\Si\to M$ of $u$ with $\Phi(I,\pd\Si)\subset\pd M$, we have (denoting $v:=\Phi'(0)\in\Ga_{\pd M}(u^*TM)$)
\begin{equation}
\label{eqn:area-energy-comparison}
\derivat{^2}{t^2}{t=0}{\rbrak{\calE - \calA}\rbrak{\Phi(t)}} = 4\int_\Si\!\abs{\mu}^2\,\ed\Si,
\end{equation}
where, splitting $v=s+\si$ into $s:=v^\perp$ and $\si:=v^\tgt$,
$$ \mu = \rbrak{
	\nabla^\tgt_\pdz \si^{\rbrak{0,1}}
	- 2e^{-2\la}\inner{s,A(u_z,u_z)}u_\zb
}\tensor\dz $$
is the \emph{infinitesimal conformal deficit} of $v$.
\end{theorem}

Let us now demystify the quantity $\mu$.
Recall that a variation $\Phi(t)$ of $u$ is conformal if and only if $\inner{\pdz\Phi(t),\pdz\Phi(t)}=0$, where the inner product $\inner{\cdot,\cdot}$ is the pull-back metric $u^*g$ on $u^*TM\to\Si$ extended complex-bilinearly (not sesquilinearly).
Let $\Phi(t)$ be a conformal variation of $u$; since we are only interested in the infinitesimal behaviour, we compute
$\pd_{t=0}\inner{\pdz\Phi(t),\pdz\Phi(t)} = 0$,
which is the case if and only if
$$ \inner{\nabla_\pdz v,\pdz u} = 0, $$
where $\nabla$ is the Levi-Civita connection on $M$.
The quantity $\mu$ is concocted so that the above equation is equivalent to the equation $\mu=0$, which according to our theorem, is the case if and only if the second variations of area and energy are equal.
To see a computation of the quantity $\mu$, see \cite[pp.~9--10]{seemungal-sharp-2026} and the surrounding discussion.
Since we do our computations locally, we will use the local quantity $\eta$, defined as $\mu=\eta\tensor\dz$.

\begin{proof}
Let $p\in\Si$ and do computations at $p$ with an orthonormal frame $E_i$ at $T_p\Si$ such that $\nabla^\tgt_{E_i}E_j = 0$ at $p$.
Recall that (suppressing the notation for summation over $E_i$),
\begin{align*}
\derivat{^2}{t^2}{t=0}\calE\rbrak{\Phi(t)}
=& \int_\Si\!
	\abs{\nabla v}^2
	- \Rm^M\rbrak{v,E_i,E_i,v}
	+ \dvg_\Si\rbrak{\nabla_vv}
	\,\ed\Si,\\
\derivat{^2}{t^2}{t=0}\calA\rbrak{\Phi(t)}
=& \int_\Si\!
	\abs{\rbrak{\nabla v}^\perp}^2
	- \Rm^M\rbrak{v,E_i,E_i,v}
	+ \rbrak{\dvg_\Si v}^2 \\
	&\qquad
	- \inner{\nabla_{E_i} v,E_j}\inner{E_i,\nabla_{E_j}v}
	+ \dvg_\Si\rbrak{\nabla_vv}
	\,\ed\Si.
\end{align*}
Then the second-order conformal deficit of $u$ is
\begin{align*}
\derivat{^2}{t^2}{t=0}{\rbrak{\calE - \calA}}\rbrak{\Phi(t)}
	&= \int_\Si\!\abs{\rbrak{\nabla v}^\tgt}^2
	- \rbrak{\dvg_\Si v}^2
	+ \inner{\nabla_{E_i}v,E_j}\inner{E_i,\nabla_{E_j}v}\,\ed\Si \\
	&=: \int_\Si\!\calC(u)\sbrak{v,v}\,\ed\Si.
\end{align*}
Splitting in the now-familiar way $v=s+\si$, where $s\in\Ga(\nu\Si)$ and $\si\in\xi$, and moreover using the following easily-checked formulae:
\begin{align*}
-\rbrak{\dvg_\Si v}^2 &=
	- \rbrak{\dvg_\Si\si}^2
	- \inner{s,H}^2
	+ 2\inner{s,H}\dvg_\Si\si
\\
\abs{\rbrak{\nabla v}^\tgt}^2 &=
	\inner{s,A(E_i,E_j)}^2
	+ \abs{\nabla^\tgt\si}^2
	- 2\inner{s,A(E_i,\nabla^\tgt_{E_i}\si)}
\\
\inner{\nabla_{E_i}v,E_j}\inner{E_i,\nabla_{E_j}v} &=
	\inner{s,A(E_i,E_j)}^2
	- 2\inner{s,A(E_i,\nabla^\tgt_{E_i}\si)}
	+ \inner{\nabla_{E_i}\si,E_j}\inner{E_i,\nabla_{E_j}\si},
\end{align*}
we have
\begin{align}
\calC(u)[v,v]
&= 2\inner{s,A(E_i,E_j)}^2
	- \inner{s,H}^2
	\nonumber
\\&\qquad
+ \abs{\nabla^\tgt\si}^2
	- \rbrak{\dvg_\Si\si}^2
	+ \inner{\nabla_{E_i}\si,E_j}\inner{E_i,\nabla_{E_j}\si}
	\tag{$\calC_2$}\label{eqn:C2}
\\&\qquad
- 4\inner{s,A(E_i,\nabla^\tgt_{E_i}\si)}
	+ 2\inner{s,H}\dvg_\Si\si.
	\tag{$\calC_3$}\label{eqn:C3}
\end{align}

\paragraph{Term \eqref{eqn:C2}.}
A short computation reveals that
$$
\dvg_\Si\rbrak{\nabla^\tgt_\si\si}
= \dvg_\Si\rbrak{\dvg_\Si\si\cdot\si}
- \rbrak{\dvg_\Si\si}^2
+ \Ric^\Si(\si,\si)
+ \inner{\nabla^\tgt_{E_j}\si,E_i}\inner{\nabla^\tgt_{E_i}\si,E_j},
$$
giving
$$
\eqref{eqn:C2}
= \abs{\nabla^\tgt\si}^2
- \Ric^\Si(\si,\si)
+ \dvg_\Si\rbrak{\nabla_\si^\tgt\si - \dvg_\Si\si\cdot\si}.
$$

\paragraph{Term \eqref{eqn:C3}.}
Another short computation, using the Codazzi--Mainardi equation, reveals that
\begin{equation*}
\Rm^M\rbrak{\si,E_i,E_i,s}
= \nabla_\si\inner{H,s}
- \inner{H,\nabla^\perp_\si,s}
- \nabla_{E_i}\inner{A(\si,E_i),s}
+ \inner{A(\si,E_i),\nabla^\perp_{E_i}s}
+ \inner{A(\nabla^\tgt_{E_i}\si,E_i),s},
\end{equation*}
giving
\begin{multline*}
\eqref{eqn:C3}
= - 2\inner{s,A(E_i,\nabla^\tgt_{E_i}\si)}
	+ 2\nabla_\si\inner{H,s}
	- 2\Rm^M\rbrak{\si,E_i,E_i,s}
	- 2\inner{H,\nabla^\perp_\si,s}
\\
	- 2\nabla_{E_i}\inner{A(\si,E_i),s}
	+ 2\inner{A(\si,E_i),\nabla^\perp_{E_i}s}
	+ 2\inner{s,H}\dvg_\Si\si.
\end{multline*}
Now, define (temporarily) the one-form $\om(\xi) := \inner{\si,\xi}\inner{H,s} - \inner{A(\si,\xi),s}$; then
$$
\dvg_\Si\om
= \dvg_\Si\si\cdot\inner{H,s}
+ \nabla_\si\inner{H,s}
- \nabla_{E_i}\inner{A(\si,E_i),s},
$$
and so
\begin{multline*}
\eqref{eqn:C3}
= - 2\inner{s,(A(E_i,\nabla^\tgt_{E_i}\si)}
- 2\Rm^M\rbrak{\si,E_i,E_i,s}
- 2\inner{H,\nabla^\perp_\si,s}
+ 2\inner{A(\si,E_i),\nabla^\perp_{E_i},s}
\\
+ \dvg_\Si\rbrak{2\sigma\inner{H,s}
	+ 2\rbrak{\nabla_\si s}^\tgt}.
\end{multline*}
Collecting these computations together, we have
\begin{align}
\calC(u)[v,v]
&= 2\inner{s,A(E_i,E_j)}^2
	- \inner{s,H}^2
		\tag{$\calC_1$}\label{eqn:C1}
\\&\qquad
	+ \abs{\nabla^\tgt\si}^2
	- \Ric^\Si(\si,\si)
		\tag{$\calC_2a$}\label{eqn:C2a}
\\&\qquad
	- 2\inner{s,A(E_i,\nabla^\tgt_{E_i}\si)}
	- 2\Rm^M(\si,E_i,E_i,s)
	- 2\inner{H,\nabla^\perp_\si s}
	+ 2\inner{A(\si,E_i),\nabla^\perp_{E_i}s}
		\tag{$\calC_3a$}\label{eqn:C3a}
\\&\qquad
	+ \dvg_\Si\rbrak{
		\nabla^\tgt_\si\si
		- \dvg_\Si\si\cdot\si
		+ 2\inner{H,s}\si
		+ 2\rbrak{\nabla_\si s}^\tgt
		}.
		\tag{$\calC_\pd$}\label{eqn:Cb}
\end{align}

At this point, we complexify.
Let $z=x+iy$ be local holomorphic co-ordinates about $p\in\Si$; then $u^*g=e^{2\la}\dz\dzb$ for some conformal factor $\la$.
We have, as usual $\pd_z = \frac{1}{2}\rbrak{\pd_x-i\pd_y}$, \etc, and so $\abs{\pd_z}^2 = \frac{1}{2}e^{2\la}$, \etc

\paragraph{Term \eqref{eqn:C1}.}
Now, a couple of straightforward computations show that
$$ \eqref{eqn:C1} = 16e^{-4\la}\abs{\inner{s,A(\pd_z,\pd_z)}}^2. $$

\paragraph{Term \eqref{eqn:C2a}.}
For term \eqref{eqn:C2a}, note first that $4e^{-2\la}\abs{\nabla^\tgt_{\pd_z}\si}^2 = \abs{\nabla^\tgt\si}^2$, so
$$ \abs{\nabla^\tgt\si}^2 = 4e^{-2\la}\rbrak{\abs{\nabla^\tgt_\pdz\si^{(1,0)}}^2 + \abs{\nabla^\tgt_{\pd_z}\si^{(0,1)}}^2}. $$
Also, using the Jacobi identity and expanding the curvature tensor, we have that
\begin{align*}
\frac{1}{4}e^{2\la}\Ric^\Si(\si,\si)
&= \Rm^\Si(\si,\pdz,\pdzb,\si) \\
&= \Rm^\Si\rbrak{\si^{(0,1)},\pdz,\pdzb,\si^{(1,0)}} \\
&= {\inner{\nabla^\tgt_\pdzb\nabla^\tgt_\pdz\si^{(0,1)},\si^{(1,0)}}
	- \inner{\nabla^\tgt_\pdz\nabla^\tgt_\pdzb\si^{(0,1)},\si^{(1,0)}}} \\
&= \abs{\nabla^\tgt_\pdz\si^{(1,0)}}^2
	- \abs{\nabla^\tgt_\pdz\si^{(0,1)}}^2
	+ \pdzb\inner{\nabla^\tgt_\pdz\si^{(0,1)},\si^{(1,0)}}
	- \pdz\inner{\nabla^\tgt_\pdzb\si^{(0,1)},\si^{(1,0)}}.
\end{align*}
Therefore,
\begin{equation*}
\eqref{eqn:C2a}
= 8e^{-2\la}\abs{\nabla^\tgt_\pdz\si^{(0,1)}}^2
- 4e^{-2\la}\pdzb\inner{\nabla^\tgt_\pdz\si^{(0,1)},\si^{(1,0)}}
+ 4e^{-2\la}\pdz\inner{\nabla^\tgt_\pdzb\si^{(0,1)},\si^{(1,0)}}.
\end{equation*}
Now, it is easy to see that
\begin{equation*}
\abs{\eta}^2
= \abs{\nabla^\tgt_\pdz\si^{(0,1)}}^2
+ 2e^{-2\la}\abs{\inner{s,A(\pdz,\pdz)}}^2
- \inner{s,A(\nabla^\tgt_\pdzb\si^{(1,0)},\pdz)}
- \inner{s,A(\nabla^\tgt_\pdz\si^{(0,1)},\pdzb)},
\end{equation*}
giving
\begin{align*}
\eqref{eqn:C2a}
&= 8e^{-2\la}\abs{\eta}^2
- 16e^{-4\la}\abs{\inner{s,A(\pdz,\pdz)}}^2
\\&\qquad
+ 8e^{-2\la}\inner{s,A(\nabla^\tgt_\pdzb\si^{(1,0)},\pdz)}
+ 8e^{-2\la}\inner{s,A(\nabla^\tgt_\pdz\si^{(0,1)},\pdzb)}
\\&\qquad
- 4e^{-2\la}\pdzb\inner{\nabla^\tgt_\pdz\si^{(0,1)},\si^{(1,0)}}
+ 4e^{-2\la}\pdz\inner{\nabla^\tgt_\pdzb\si^{(0,1)},\si^{(1,0)}}.
\end{align*}

\paragraph{Term \eqref{eqn:C3a}.}
Note firstly that
\begin{equation*}
\inner{A(\si,E_i),\nabla^\perp_{E_i}s}
= \frac{1}{2}\inner{H,\nabla^\perp_\si s}
+ 2e^{-2\la}\rbrak{\inner{A(\si^{(1,0)},\pdz),\nabla^\perp_\pdzb s}
	+ \inner{A(\si^{(0,1)},\pdzb),\nabla^\perp_\pdz s}}.
\end{equation*}
Also, a long but straightforward computation (using Jacobi's identity and splitting terms into their tangential and normal parts) shows that
\begin{align*}
\Rm^M(\si,E_i,E_i,s)
&=
- \frac{1}{2}\inner{H,\nabla^\perp_\si s}
\\&\quad
+ 2e^{-2\la}\inner{A\rbrak{\pdzb,\nabla^\tgt_\pdz\si^{(0,1)}},s}
- 2e^{-2\la}\inner{A\rbrak{\pdzb,\nabla^\tgt_\pdz\si^{(1,0)}},s}
\\&\quad
+ 2e^{-2\la}\inner{A\rbrak{\pdz,\nabla^\tgt_\pdzb\si^{(1,0)}},s}
- 2e^{-2\la}\inner{A\rbrak{\pdz,\nabla^\tgt_\pdzb\si^{(0,1)}},s}
\\&\quad
+ 2e^{-2\la}\inner{A\rbrak{\pdzb,\si^{(0,1)}},\nabla^\perp_\pdz s}
+ 2e^{-2\la}\inner{A\rbrak{\pdz,\si^{(1,0)}},\nabla^\perp_\pdzb s}
\\&\quad
+ 2e^{-2\la}\pdzb\rbrak{
	\inner{A\rbrak{\pdz,\si^{(0,1)}},s}
	- \inner{A\rbrak{\pdz,\si^{(1,0)}},s}}
\\&\quad
+ 2e^{-2\la}\pdz\rbrak{
	\inner{A\rbrak{\pdzb,\si^{(1,0)}},s}
	- \inner{A\rbrak{\pdzb,\si^{(0,1)}},s}}.
\end{align*}
We also have
$$
\inner{s,A(E_i,\nabla^\tgt_{E_i}\si)}
= 2e^{-2\la}\inner{s,A\rbrak{\pdz,\nabla^\tgt_\pdzb\si}}
+ 2e^{-2\la}\inner{s,A\rbrak{\pdzb,\nabla^\tgt_\pdz\si}},
$$
giving
\begin{align*}
\eqref{eqn:C3a}
&= -8e^{-2\la}\inner{A\rbrak{\pdzb,\nabla^\tgt_\pdz\si^{(0,1)}},s}
+ 8e^{-2\la}\inner{A\rbrak{\pdz,\nabla^\tgt_\pdzb\si^{(1,0)}},s}
\\&\qquad
- 4e^{-2\la}\pdzb\inner{A\rbrak{\pdz,\si^{(0,1)}-\si^{(1,0)}},s}
- 4e^{-2\la}\pdz\inner{A\rbrak{\pdzb,\si^{(1,0)}-\si^{(0,1)}},s}.
\end{align*}
Using therefore our computed expressions for \eqref{eqn:C1}, \eqref{eqn:C2a}, \eqref{eqn:C3a}, we have
\begin{align*}
\calC(u)[v,v]
&= 8e^{-2\la}\abs{\eta}^2
+ 4e^{-2\la}\pdz\rbrak{\inner{\nabla^\tgt_\pdzb\si^{(0,1)},\si^{(1,0)}}
	- \inner{A\rbrak{\pdzb,\si^{(1,0)}-\si^{(0,1)}},s}}
\\&\qquad
- 4e^{-2\la}\pdzb\rbrak{\inner{\nabla^\tgt_\pdz\si^{(0,1)},\si^{(1,0)}}
	- \inner{A\rbrak{\pdz,\si^{(1,0)}-\si^{(0,1)}},s}}
\\&\qquad
+ \dvg_\Si\rbrak{
		\nabla^\tgt_\si\si
		- \dvg_\Si\si\cdot\si
		+ 2\inner{H,s}\si
		+ 2\rbrak{\nabla_\si s}^\tgt
		} \\
&= 8e^{-2\la}\abs{\eta}^2
+ 4e^{-2\la}\pdz T(\pdzb)
- 4e^{-2\la}\pdzb T(\pdz)
\\&\qquad
+ \dvg_\Si\rbrak{
		\nabla^\tgt_\si\si
		- \dvg_\Si\si\cdot\si
		+ 2\inner{H,s}\si
		+ 2\rbrak{\nabla_\si s}^\tgt
		},
\end{align*}
where we have defined the covector
$$
T(\zeta)
= \inner{\nabla_\zeta^\tgt\si^{(0,1)},\si^{(1,0)}}
- \inner{A\rbrak{\zeta,\si^{(1,0)}-\si^{(0,1)}},s}.
$$
It is worth pausing for a moment to see where we are in the computation: the only term which is not integrated to a boundary integral is the first one.

Focus on the "$T$"-terms, and integrate:
\begin{align*}
\integ{\Si}{}{4e^{-2\la}\rbrak{\pdz T(\pdzb) - \pdzb T(\pdz)}}{\Si}
&= \geomint{\Si}{2i\ed\rbrak{T(\pdzb)\dzb + T(\pdz)\dz}} \\
&= \geomint{\pd\Si}{2i\rbrak{T(\pd_x)\ed x + T(\pd_y)\ed y},} \\
\intertext{but along $\pd\Si$, we have $\pd_\tau=e^{-\la}\pd_x$ and $\pd_y=-e^{\la}\rmn$, so}
&= \geomint{\pd\Si}{2iT(\pd_\tau)}\,\ed\tau.
\end{align*}
(At first glance, it looks like we potentially have a problem, in that the above integral looks like it is purely imaginary.
Note however, that since $T(\zeta)+\overline{T}(\zeta)=\nabla_\zeta\abs{\sigma^{(0,1)}}^2\in\smooth{\Si,\R}$, we have that $T\in\Ga(T^*\Si\tensor i\R)$ is a purely imaginary covector as long as $\zeta$ is a real tangent vector.
Therefore the above integral is in fact real.)

We arrive then (grouping the terms into "$\si$-only" and "mixed $\si,s$" terms) at
\begin{multline*}
\geomint{\Si}{\calC(u)[v,v]\,\ed\Si}
= \geomint{\Si}{8\abs{\eta}^2\,\ed\Si}
\\
+ \int_{\pd\Si}\!
\Big(
	2i\inner{\nabla_{\pd_\tau}\si^{(0,1)},\si^{(1,0)}}
	+ \inner{\nabla^\tgt_\si\si,\rmn}
	- \inner{\si,\rmn}\dvg_\Si\si
\\
	- 2i\inner{A\rbrak{\pd_\tau,\si^{(1,0)}-\si^{(0,1)}},s}
	+ 2\inner{H,s}\inner{\si,\rmn}
	- 2\inner{A(\si,\rmn),s}
\Big)\,\ed\tau.
\end{multline*}

Since $\sigma\in\Ga(u^*TM)$, there exists some $\phi\in\smooth{\Si,\C}$ such that
\begin{align*}
\si &= \sqrt{2}e^{-\la}\rbrak{\phi\pd_z + \overline{\phi}\pd_\zb} \\
&= \sqrt{2}e^{-\la}\rbrak{\Re\rbrak{\phi}\pd_x + \Im\rbrak{\phi}\pd_y} \\
&=: fe^{-\la}\pd_x + ge^{-\la}\pd_y  \\
&= f\pd_\tau - g\rmn,
\end{align*}
where $f,g$ are appropriately-defined real-valued functions.
Moreover, we have a multitude of formulae; for instance, $\Re\phi = \frac{1}{2}\rbrak{\phi+\overline{\phi}}$, $g=\sqrt{2}\Im\phi = \frac{i}{\sqrt{2}}\rbrak{\overline{\phi}-\phi}$, $\phi=\frac{1}{\sqrt{2}}\rbrak{f+ig}$, \etc{}, which we shall use liberally in the subsequent.

\noindent\textbf{The mixed "$\si,s$" terms.}
Using the above formulae, we have
\begin{align*}
\si^{(1,0)} &= \frac{1}{2}e^{-\la}\rbrak{f\pd_x + g\pd_y + ig\pd_x - if\pd_y} \\
\si^{(0,1)} &= \frac{1}{2}e^{-\la}\rbrak{f\pd_x + g\pd_y - ig\pd_x + if\pd_y},
\end{align*}
so that $\si^{(1,0)}-\si^{(0,1)}=ig\pd_\tau + if\rmn$.
Moreover, we have
\begin{align*}
- 2i\inner{A\rbrak{\si^{(1,0)}-\si^{(0,1)}},s}
	&=
	2g\inner{A(\pd_\tau,\pd_\tau),s}
	+ 2f\inner{A(\pd_\tau,\rmn),s}
\\
2\inner{H,s}\inner{\si,\rmn}
	&=
	- 2g\inner{A(\pd_\tau,\pd_\tau),s}
	- 2g\inner{A(\rmn,\rmn),s}
\\
- 2\inner{A(\si,\rmn),s}
	&=
	-2f\inner{A(\pd_\tau,\rmn),s}
	+ 2g\inner{A(\rmn,\rmn),s}.
\end{align*}
Using these computations, we see that the mixed terms in fact vanish:
$$
- 2i\inner{A\rbrak{\si^{(1,0)}-\si^{(0,1)}},s}
+ 2\inner{H,s}\inner{\si,\rmn}
- 2\inner{A(\si,\rmn),s} = 0.
$$

\noindent\textbf{The unmixed "$\si$-only" terms.}
Using the expression above of $\si^{(1,0)}$ and $\si^{(0,1)}$, one computes that
\begin{align*}
2i\inner{\nabla_{\pd_\tau}\si^{(0,1)},\si^{(1,0)}}
&=
\frac{1}{2}i \pd_\tau\rbrak{f^2 + g^2}
+ f\pd_\tau g
- g \pd_\tau f \\
&=
\frac{1}{2}i \pd_\tau\abs{\si}^2
+ f\pd_\tau g
- g \pd_\tau f. \\
\end{align*}
We also have
$\inner{\nabla^\tgt_\si\si,\rmn} = -f\pd_\tau g + g\pd_\rmn g$
and
$-\inner{\si,\rmn}\dvg_\Si\si = g\pd_\tau f - g\pd_\rmn g$.
We see therefore that the unmixed terms amount to
$$
2i\inner{\nabla_{\pd_\tau}\si^{(0,1)},\si^{(1,0)}}
+ \inner{\nabla^\tgt_\si\si,\rmn}
- \inner{\si,\rmn}\dvg_\Si\si
=
\frac{1}{2}i \pd_\tau\abs{\si}^2,
$$
which, after integrating along $\pd\Si$ vanishes.
There are no more terms: we have simply
$$
\geomint{\Si}{\calC(u)[v,v]\,\ed\Si}
=
8\geomint{\Si}{\abs{\eta}^2\,\ed\Si.}
$$
\end{proof}

This comparison of second variations gives the following comparison of indices, taking us half-way to the main result.

\begin{theorem}
\label{thm:index-area-energy-comparison}
Let $u:\Si\to M$ be capillary \cmc{} surface of genus $g$, interior branching order $b$, boundary branching order $d$, and $m$ boundary components.
Then
\begin{equation}
\label{eqn:index-area-energy-comparison}
i_\Si \leq \indx\rbrak{Q_\calE} + r
\end{equation}
where
$$
r = 
\begin{cases}
6g - 6 + 3m - 2b - d &\text{if }2b+d<4g-4+2m,\\
4g - 2 + 2m - 2b + 2\floor{-d/2} &\text{if } 4g-4+2m\leq2b+d\leq8g-8+4m,\\
0 &\text{if }8g-8+4m<2b+d.
\end{cases}
$$
\end{theorem}
\begin{proof}
Let $V$ be the negative eigenspace of $Q$, so that $\dim V=i_\Si$.
Let $\phi:V\to\xi^{0,1}\tensor(T^*\Si)^{1,0}$ be the linear map
$\phi(f):=2e^{-2\la}\inner{f\nu,A(u_z,u_z)}u_{\zb}\tensor\dz$, and define $s=f\nu$.
Recall then that the idea is to find a $\si$, depending of course on $s$, such that the variation giving $\si+s$ remains in the conformal class.

Indeed, the equation $\mu=0$ from Theorem \ref{thm:area-energy-comparison} is then written $D\si^{0,1} = \phi(s)$, where $D:=\dz\tensor\nabla^{\tgt}_{\pdz}$ is defined globally on $\Ga(\xi^{0,1})\to\Ga(\wedge^{1,0}\tensor\xi^{0,1})$.
Consider the boundary value problem
\begin{equation}
\label{eqn:zero-inf-conform-deficit}
\left\{
\begin{aligned}
D\si^{0,1} &= \phi(s) &\qquad\text{on }\Si \\
\Im(\inner{\si^{0,1},\pdz}) &= -\cot\theta\inner{s,\nu} &\qquad\text{on }\pd\Si,
\end{aligned}
\right.
\end{equation}
where the boundary condition is precisely the condition that $\si+s$ be a section not just of $u^*TM$, but also of $u|_{\pd\Si}T\pd M$ on the boundary.\footnote{Indeed, $\si+s\in\Ga_{\pd M}(u^*TM)$ if and only if $\inner{\si+s,N}=0$, which is the case if and only if $\inner{\si,\rmn}=\cot\theta\inner{s,\nu}$.
But, since locally $\si=\Re(g)u_x-\Im(g)u_y$ (meaning of course that $\si^{0,1}=gu_\zb$, for some function $g$), and in particular on the boundary $\si=\Re(g)u_\tau - \Im(g)u_\rmn$, we see that $\si+s\in\Ga_{\pd M}(u^*TM)$ if and only if $\Im(g)=-\cot\theta\inner{s,\nu}$.
In the free-boundary case we simply have $\Im(g)=0$.}
We wish to transform the boundary condition from an affine-linear condition (as it is now) into a linear condition, so that the boundary condition define a totally real subbundle, and the PDE problem be described as a $\dbar$-problem for sections which lie in a totally real subbundle at the boundary.

For a given $s$, then, fix any real vector field $\zeta_s\in\Ga(\xi)$ such that $\Im(\inner{\zeta_s^{0,1},\pdz})=-\cot\theta\inner{s,\nu}$ on $\pd\Si$.
If $\chi=\si+\zeta_s$, then on the boundary $\Im(\inner{\chi^{0,1},\pdz})=0$, which is a linear condition, and equivalent to $\chi\in\Ga(F)$ where $F\to\pd\Si$ is the totally real subbundle of $\xi^{0,1}$ whose sections are given locally by $g\pdzb$, where $g$ is a \emph{real} function.
Moreover, $\si$ solves \eqref{eqn:zero-inf-conform-deficit} above if and only if $\chi$ solves
\begin{equation}
\label{eqn:zero-inf-conform-deficit-transformed}
\left\{
\begin{aligned}
D\chi^{0,1} &= \phi(s) + D\zeta_s^{0,1} &\qquad\text{on }\Si \\
\chi|_{\pd\Si} &\in F. &\end{aligned}
\right.
\end{equation}
Define for convenience $\psi(s):=\phi(s)+D\zeta_s^{0,1}$.

By the Fredholm alternative, \eqref{eqn:zero-inf-conform-deficit-transformed} can be solved if and only if $\psi(s)\perp\ker D^*$, where
$$ D^*:\Ga_F(\xi^{0,1}\tensor\wedge^{1,0})\to\Ga(\xi^{0,1}). $$
We may accordingly decompose
$V = \ker\psi\oplus V_1\oplus V_2$
where
$$ V_1=\cbrak{s\stbar\psi(s)\in\rbrak{\ker D^*}^\perp}
\qquad\text{and}\qquad
V_2=\cbrak{s\stbar\psi(s)\in\ker D^*}.
$$
According to our splitting then, given a basis $s_1,\ldots,s_n$ for $\ker\psi\oplus V_1$, we have corresonding solutions $\chi_1,\ldots,\chi_n$ to \eqref{eqn:zero-inf-conform-deficit-transformed}, and therefore corresponding  solutions $\si_1,\ldots\si_n$ to \eqref{eqn:zero-inf-conform-deficit}, and so by Theorem \ref{thm:area-energy-comparison} we have that for each $i=1,\ldots,n$,
$$ Q\rbrak{\inner{s_i,\nu},\inner{s_i,\nu}} = Q_{\calE}\rbrak{s_i+\si_i,s_i+\si_i}. $$
Straightforwardly then, we have
$$ \dim\rbrak{\ker\psi\oplus V_1}
= \indx\rbrak{Q_{\calE}\stbar{\ker\psi\oplus V_1}}
\leq \indx{Q_\calE}.
$$

It remains to show that $\dim_{\R}\ker(D^*)\leq r$.
An integration by parts shows that $D^*=i\star\overline{\pd}:\Ga_F(\xi^{0,1}\tensor\wedge^{1,0})\to\Ga(\xi^{0,1})$, where $\star$ is the Hodge star.
Therefore $\ker(D^*)=H^0_F\rbrak{\xi^{0,1}\tensor\kappa}$, the space of holomorphic sections of the bundle $\xi^{0,1}\tensor\kappa$ which lie in $F$ at the boundary.
(Here, $\kappa=\wedge^{1,0}$ is the canonical bundle of $\Si$.)
We compute this using the Riemann--Roch theorem as presented in \cite{mcduff-salamon-2012}.

Firstly, the index formula from \cite[Thm.~C.1.10]{mcduff-salamon-2012} applied to the operator $\dbar:\Ga_F(\xi^{1,0})\to\Ga(\xi^{1,0}\tensor\overline{\kappa})$ (note that this $\dbar$ operator has different domain and range to that above) gives
$$
\dim\ker\dbar - \dim\coker\dbar = \chi(\Si) + \mu(\xi^{1,0},F),
$$
where $\chi(\Si)$ and $\mu(\xi^{1,0},F)$ are the Euler characteritic and boundary Maslov index respectively.
Secondly, we have that $\ker\dbar=H^0_F(\xi^{1,0})$, so $\dim\ker\dbar=2h^0(\xi^{1,0})$.\footnote{We denote by $h^0$ the complex dimension of the space, so the real dimension is $2h^0$.}
Thirdly, the equivalence \cite[(C.1.3)]{mcduff-salamon-2012} gives that $\coker\dbar\iso\ker(\dbar^*:\Ga_F(\xi^{1,0}\tensor\overline{\kappa})\to\Ga(\xi^{1,0}))$, and similarly to above we have $\dbar^*=-i\star\pd:\Ga_F(\xi^{1,0}\tensor\overline{\kappa})\to\Ga(\xi^{1,0})$; taking moreover the complex conjugate we have $\coker\dbar\iso H_F^0(\xi^{0,1}\tensor\kappa)=\ker(D^*)$, so our Riemann--Roch formula is
\begin{equation}
\label{eqn:rrformula}
2h^0(\xi^{1,0})
- 2h^0(\xi^{0,1}\tensor\kappa)
=
\chi(\Si) + \mu(\xi^{1,0},F).
\end{equation}

By \cite[Thm.~C.1.10~\textit{(iii)}]{mcduff-salamon-2012}, there are three cases.
Firstly, if $\mu(\xi^{1,0},F)<0$ then $\dbar:\Ga_F(\xi^{1,0})\to\Ga(\xi^{1,0}\tensor\overline{\kappa})$ is injective, for which reason $h^0(\xi^{1,0})=0$, and by \eqref{eqn:rrformula} we have
$
2h^0(\xi^{0,1}\tensor\kappa)
=
-\chi(\Si)-\mu(\xi^{0,1},F)
$.
Secondly, if $\mu(\xi^{1,0},F)+2\chi(\Si)>0$, then $\dbar:\Ga_F(\xi^{1,0})\to\Ga(\xi^{1,0}\tensor\overline{\kappa})$ is surjective, so $h^0(\xi^{0,1}\tensor\kappa)=0$.
Finally, suppose that $0\leq\mu(\xi^{1,0},F)\leq -2\chi(\Si)$.
We may as well assume that $h^0(\xi^{0,1}\tensor\kappa)>0$.
If we also assume that $h^0(\xi^{1,0})>0$, then the bundle is special, so applying Clifford's theorem gives the estimate (where $\Tilde{\Si}$ is the double $\Si\cup\hat{\Si}$, as in \cite{lima-2022}, to which we refer for full details),
$$
2h^0(\xi^{0,1}\tensor\kappa)
\leq
\inner{c_1(\xi^{0,1}\tensor\kappa),\Tilde{\Si}} + 2.
$$
But even in the case that $h^0(\xi^{1,0})=0$, this estimate holds: indeed, by \eqref{eqn:rrformula} we have $2h^0(\xi^{0,1}\tensor\kappa)=-\chi(\Si)-\mu(\xi^{1,0},F)$; combining this with the fact that (by assumption) $-\chi(\Si)\geq0$ and that $\inner{c_1(\xi^{1,0}),\tilde{\Si}}=-\mu(\xi^{1,0},F)-2\chi(\Si)$ gives $2h^0(\xi^{0,1}\tensor\kappa) \leq \inner{c_1(\xi^{0,1}\tensor\kappa),\Tilde{\Si}}$, from which the claimed estimate easily follows.

Finally, we compute the boundary Maslov index.
By \cite[Thm.~C.3.10]{mcduff-salamon-2012} and symmetry, we have that
\begin{align*}
\mu(\xi^{1,0},F) &= \inner{c_1(\xi^{1,0}),[\Tilde{\Si}]} \\
&=
\inner{c_1(T^{1,0}\Tilde{\Si}),[\Tilde{\Si}]}
+ \inner{c_1(D),[\Tilde{\Si}]}
\\
&=
\chi(\Tilde{\Si}) +2b +d
\\
&= 2\chi(\Si) +2b+d.
\end{align*}

Putting these cases together, and computing the various Euler characteristics and Chern classes, we have
\begin{equation}
2h^0(\xi^{0,1}\tensor\kappa)
\leq
\begin{cases}
6g - 6 + 3m - 2b - d &\text{if }2b+d<4g-4+2m,\\
4g - 2 + 2m - 2b + 2\floor{-d/2} &\text{if } 4g-4+2m\leq2b+d\leq8g-8+4m,\\
0 &\text{if }8g-8+4m<2b+d.
\end{cases}
\end{equation}

\end{proof}

\section{Energy Index Bound}

To obtain our energy index bound, we will be analysing the bundle heat kernel $\overline{K}:\Si\times\Si\times(0,\infty)\to u^*TM$ defined as
\begin{equation*}
\left\{
\begin{aligned}
\rbrak{\pderiv{}{t}-\overline{\Delta}}\overline{K}(x,y,t) &= 0\qquad\text{on }\Si\times\Si\times(0,\infty) \\
\rbrak{\nabla_\rmn - B\sin\theta}\overline{K}(x,y,t) &= 0\qquad\text{on }\pd\Si\times\Si\times(0,\infty).
\end{aligned}
\right.
\end{equation*}
Actually, we will estimate the bundle heat kernel by the following heat kernel on functions, $K:\Si\times\Si\times(0,\infty)\to\R$ defined as
\begin{equation*}
\left\{
\begin{aligned}
\rbrak{\pderiv{}{t}-\Delta}K(x,y,t) &= 0\qquad\text{on }\Si\times\Si\times(0,\infty) \\
\pderiv{K}{\rmn}(x,y,t) &= 0\qquad\text{on }\pd\Si\times\Si\times(0,\infty).
\end{aligned}
\right.
\end{equation*}

\begin{theorem}
\label{thm:energy-index-bound}
There exists a universal constant $C>0$ such that for any Riemannian manifold $M$ with boundary, isometrically embedded into some Euclidean space $\R^d$, and for any oriented Riemann surface $\Si$ of genus $g$ with $m\neq0$ boundary components, if $u:\Si\to M$ is a $\theta$-capillary $h$-\cmc{} surface, then
$$ i_{\calE} \leq C\rbrak{1+\frac{1}{\sin\theta}}^2(J^2+B^2+h^2)\calA(u) $$
where $J$ is the sup-norm of the largest eigenvalue of the second fundamental form of the embedding $M\into\R^d$, and $B=\norm{A^{\pd M}}_\infty$.
\end{theorem}

\begin{proof}
Estimating as in \cite[Proof of Thm.~3.2]{seemungal-sharp-2026}, we have
\begin{align*}
Q_\calE(v,v)
&= \geomint{\Si}{
	\abs{\nabla v}^2
	- \Rm^M(v,E_i,E_i,v)
	+ h\rbrak{
		\Vol^M(v,\nabla_{e^{-\la}u_x}v,e^{-\la}u_y)
		+ \Vol^M(v,e^{-\la}u_x,\nabla_{e^{-\la}u_y}v)
	}
\,\ed\Si
}
\\&\qquad
- \geomint{\pd\Si}{
\sin\theta\inner{A^{\pd M}(\dot{\ga},{\dot{\ga}}),N}\abs{v}^2
\,\ed\tau
}
\\
&\geq
\frac{1}{2}
\geomint{\Si}{
\abs{\nabla v}^2
- (4J^2+2h^2)\abs{v}^2
\,\ed\Si
}
- \geomint{\pd\Si}{
B\sin\theta\abs{v}^2
\,\ed\tau
}
\\
&=
\frac{1}{2}
\geomint{\Si}{
\inner{-\rbrak{\overline{\Delta}+4J^2+2h^2}v,v}
\,\ed\Si
}
+
\geomint{\pd\Si}{
\inner{\nabla_\rmn v,v}
- B\sin\theta\abs{v}^2
\,\ed\tau.
}
\\
&\geq
\frac{1}{2}
\geomint{\Si}{
\inner{-\rbrak{\overline{\Delta}+4J^2+2h^2+2B^2}v,v}
\,\ed\Si
}
+
\geomint{\pd\Si}{
\inner{\nabla_\rmn v,v}
- B\sin\theta\abs{v}^2
\,\ed\tau.
}
\end{align*}
Therefore, if $\overline{\la}$ are the eigenvalues corresponding to
\begin{align*}
-\overline{\Delta}v &= \overline{\la}v\qquad\textrm{on }\Si\\
\rbrak{\nabla_\rmn - B\sin\theta}v &= 0\qquad\textrm{on }\pd\Si,
\end{align*}
then
$
i_{\calE}+n_{\calE}
\leq
\#\cbrak{\overline{\la}\leq 4J^2+2h^2+2B^2}
$
(counted of course with multiplicity).
Now, for all $t>0$, we have
\begin{equation}
\#\cbrak{\overline{\la}\leq 4J^2+2h^2+2B^2}
\leq
\sum_i e^{-t(\overline{\la}_i-(4J^2+2h^2+2B^2))}
\label{eqn:count-eval-est}
=
e^{(4J^2+2h^2+2B^2)t}\overline{h}(t),
\end{equation}
where $\overline{k}(t)=\sum_i e^{\overline{\la}_i t}$ is the trace of the heat kernel $\overline{K}$.
Arguing as in \cite[Thm~2.1]{urakawa-1987} and \cite[p.~15]{lima-2022} gives that $k(t)\leq3\overline{k}(t)$, where $k(t)$ is now the trace of the heat kernel $K$.
Our goal has therefore now changed to finding bounds on $k(t)$.

Indeed, take the Sobolev inequality from Lemma \ref{lem:sobolev},
\begin{align*}
\frac{\sqrt{2\pi}\sin\theta}{\sin\theta+1}\rbrak{\int_\Si\!f^2}^{1/2}
&\leq
\geomint{\Si}{\abs{\nabla f}}
+ \rbrak{\sqrt{h^2+4J^2} + \frac{2}{\rho(\sin\theta+1)}}\geomint{\Si}{\abs{f}}
\\
&=:
\geomint{\Si}{\abs{\nabla f}}
+ L\geomint{\Si}{\abs{f}},
\end{align*}
where we have defined for convenience $L:=\sqrt{h^2+4J^2}+2/(\rho(\sin\theta+1))$.
Replacing $f$ with $f^2$, and applying the Peter--Paul inequality
$2\abs{f}\abs{\nabla f}\leq\frac{\abs{\nabla f}^2}{\delta L} + \delta Lf^2$
gives
\begin{equation*}
\rbrak{\frac{\sqrt{2\pi}\sin\theta}{\sin\theta+1}}
\rbrak{\geomint{\Si}{f^4}}^{1/2}
\leq
\frac{1}{\delta L}\geomint{\Si}{\abs{\nabla f}^2}
+ L(\delta+1)\geomint{\Si}{\abs{f}^2}.
\end{equation*}
The interpolation inequality
$\norm{f}^3_{L^2} \leq \norm{f}^2_{L^4} \norm{f}_{L^1}$
then gives
\begin{equation}
\label{eqn:tracebound1}
\rbrak{\frac{\sqrt{2\pi}\sin\theta}{\sin\theta+1}}
\frac{\rbrak{\geomint{\Si}{f^2}}^{3/2}}{\geomint{\Si}{\abs{f}}}
\leq
\frac{1}{\delta L}\geomint{\Si}{\abs{\nabla f}^2}
+ L(\delta+1)\geomint{\Si}{\abs{f}^2}.
\end{equation}
Getting back to bounding $k$, we set $f(y):=K(x,y,t/2)$, and use the various standard properties of the heat kernel; in particular, we calculate that
\begin{align*}
\pderiv{}{t}\rbrak{K(x,x,t)}
&= \geomint{\Si}{K(x,y,t/2)\rbrak{\pd_t K}(x,y,t/2)\,\ed y} \\
&= \geomint{\Si}{K(x,y,t/2)\rbrak{\Delta_y K}(x,y,t/2)\,\ed y} \\
&= -\geomint{\Si}{\abs{\nabla K}^2(x,y,t/2)\,\ed y}
+ \geomint{\pd\Si}{K(x,y,t/2)\pderiv{K}{\rmn}(x,y,t/2)\,\ed y} \\
&= -\geomint{\Si}{\abs{\nabla K}^2(x,y,t/2)\,\ed y}.
\end{align*}
Furthermore, $K(x,x,t)=\geomint{\Si}{K(x,y,t/2)^2\,\ed y}$, and $\geomint{\Si}{\abs{K(x,y,t/2)}} \leq 1$, so using these formulae in \eqref{eqn:tracebound1} gives
\begin{equation*}
\frac{\sqrt{2\pi}\sin\theta}{\sin\theta+1} K(x,x,t)^{3/2}
\leq
L(\delta+1)K(x,x,t)
- \frac{1}{\delta L}\pd_t K(x,x,t).
\end{equation*}
Since $K>0$ for $t>0$, we set $\phi(t):=K(x,x,t)^{-1/2}$, yielding
\begin{equation*}
\frac{\sqrt{2\pi}\sin\theta}{\sin\theta+1}
\leq
L(\delta+1)\phi
+ \frac{2}{\delta L}\phi'.
\end{equation*}
Multiplying by the integrating factor gives
\begin{equation*}
\sqrt{\frac{\pi}{2}}
\frac{\delta L\sin\theta}{\sin\theta+1}
\exp\rbrak{\frac{1}{2}L^2\delta(\delta+1)t}
\leq
\deriv{}{t}\rbrak{
\exp\rbrak{\frac{1}{2}L^2\delta(\delta+1)t}\phi(t)
}.
\end{equation*}
Integrating (using the fact that $\phi(0)=0$) and subsequently solving for $\phi^{-2}=K(x,x,t)$ yields
\begin{equation*}
K(x,x,t)
\leq
\frac{1}{2\pi}\rbrak{1+\frac{1}{\sin\theta}}^2
(\delta+1)^2
L^2
\rbrak{
\frac{\exp\rbrak{\frac{1}{2}L^2\delta(\delta+1)t}}
{\exp\rbrak{\frac{1}{2}L^2\delta(\delta+1)t}-1}
}^2.
\end{equation*}
But of course, from \eqref{eqn:count-eval-est}, and since $k(t)=\geomint{\Si}{K(x,x,t)\,\ed\Si(x)}$, we have
$$
i_{\cal{E}}+n_{\cal{E}}
\leq
e^{\rbrak{4J^2+2h^2+2B^2}t}k(t)
=
e^{\rbrak{4J^2+2h^2+2B^2}t}\geomint{\Si}{K(x,x,t)}
$$
and so we arrive at (after some mild algebraical manipulation)
\begin{equation*}
i_{\calE} + n_{\calE}
\leq
\frac{1}{2\pi}
\rbrak{1+\frac{1}{\sin\theta}}^2
(\delta+1)^2
L^2
\calA(u)
\rbrak{
\frac{\exp\rbrak{\sbrak{\frac{1}{2}L^2\delta(\delta+1) + 2J^2 + h^2 + B^2}t}}
{\exp\rbrak{\frac{1}{2}L^2\delta(\delta+1)t}-1}
}^2.
\end{equation*}
We may optimise our inequality by infimising the right hand side over $t>0$.
Indeed, the function needing to be minimised is of the form
$f(t)= {e^{(\al+\be)t}}/\rbrak{e^{\al t} - 1}$,
whose minimal value is
$f(t_0)=(\be/\al)\rbrak{\al/\be+1}^{1+\be/\al}$.
Now, it is easy to see that there exist universal constants $c,C,c',C'>0$ such that
\begin{align*}
c(J^2+h^2+B^2)  \leq & \alpha \leq C(J^2+h^2+B^2)\\
c'(J^2+h^2+B^2) \leq & \beta \leq C'(J^2+h^2+B^2),
\end{align*}
for which reason $\alpha/\beta$ and $\beta/\alpha$ are bounded from above and below by universal constants, and hence
\begin{equation*}
i_{\calE} + n_{\calE}
\leq
C
\rbrak{1+\frac{1}{\sin\theta}}^2
(J^2+B^2+h^2)
\calA(u).
\end{equation*}
\end{proof}

\subsection{Sobolev inequality}

We will find, in proving a Sobolev inequality for boundaried branched immersions $u:\Si\to M$, that we will need to extend the normal vector field $N$ of $\pd M$ some distance into the interior of $M$.
We will desire for this extension to have as small gradient as possible, so that our Sobolev inequality might be more optimal, and so we want to extend the normal a large distance into the interior of $M$.
On the other hand, if we extend too far, then we may find the extension to be ill-defined.
It is clear then that we need to consider a modification of the notion of \emph{focal radius}, which we define below.

\begin{definition}
The \emph{focal radius} of a manifold $M$ with boundary is the largest radius $r>0$ such that every point on the boundary is touched by an interior ball of that radius; \ie
$$ \sup\cbrak{r>0 \stbar \forall y \in\pd M, \exists B(x,r)\subset M \suchthat y\in\pd B(x,r)}. $$
\end{definition}

We can now extend $N$ smoothly into $M$ in the following way: as long as we stay within a distance of the focal radius of the boundary, we can extend $N$ by parallelly transporting it along geodesic lines orthogonal to the boundary.
At the focal radius, we may encounter a non-uniqueness problem, which is resolved by making the vector field decay to zero as we approach the focal radius.
Naturally, our choice of decay factor will determine the constant in the Sobolev inequality.
In the following proposition we construct this extension.

\begin{prop}
\label{prop:extend-normal}
Suppose that $M$ is a manifold with boundary, whose outward-pointing unit normal we denote by $N$, and whose inner focal radius we denote by $\rho$.
Then for every $\eps\in(0,\rho)$, there is a vector field $X_\eps\in\Ga(M)$ such that $\abs{X_\eps}\leq 1$ and $\nabla X_\eps \leq 1/(\rho - \eps)$.
\end{prop}

\begin{proof}
Since $0<\rho-\eps<\rho$, there is a smooth function $\phi_\eps:M\to\R$ which is unity on $\pd M$, decreasing monotonically along geodesic lines emanating orthogonally from $\pd M$, zero on $M_{-(\rho-\eps)}:=\cbrak{x\in M : \dist(x,\pd M)<\rho-\eps}$ (the $(\rho-\eps)$-thinning of $M$), and moreover whose derivative is not steeper than $1/(\rho-\eps)$.
Indeed, for the very same reason, one can parallelly transport $N$ into $M\setminus M_{-(\rho-\eps)}$ along geodesic lines, emanating orthogonally from $\pd M$, obtaining a smooth vector field $\tilde{N}$ on $M\setminus M_{-(\rho-\eps)}$.
Defining then
$$ X_\eps = \begin{cases}
\phi_\eps\tilde{N} &\text{on } M\setminus M_{-(\rho-\eps)}\\
0 &\text{elsewhere}
\end{cases} $$
yields the requisite vector field.
\end{proof}

The proof of the following Sobolev inequality is based on the method due to L.~Simon (as explained in for instance the Appendix to \cite{topping-2008}).

\begin{lemma}\label{lem:sobolev}
Let $\Si$ be a surface with boundary, and $M$ a compact Riemannian manifold with boundary, embedded, thanks to the Nash isometric embedding theorem, into $\R^d$.
Suppose also that $u:\Si\to M$ is a boundaried branched immersion, and extend it to $u:\Si\to M \into\R^d$.
Then for any $f\in W^{1,1}(\Si)\cap L^1(\pd\Si)$, we have
\begin{equation*}
\sqrt{2\pi}\rbrak{\int_\Si\!f^2}^{1/2}
\leq
\rbrak{1+\frac{1}{\sin\theta}}\int_\Si\!\abs{\nabla f}
+ \rbrak{1+\frac{1}{\sin\theta}}\int_\Si\!\abs{fH}
+ \frac{2}{\rho\sin\theta}\int_\Si\!\abs{f},
\end{equation*}
where $H$ is the mean curvature vector of the immersion of $\Si$, not into $M$, but into $\R^d$.
\end{lemma}
\begin{proof}
We begin with two applications of the first variation formula
\begin{equation}\label{eqn:sob-firstvar}
\int_\Si\!\dvg_\Si{v} = \int_\Si\!-\inner{v,H} + \int_{\pd\Si}\!\inner{v,\rmn}
\end{equation}
to two different functions.
The following two computations will be useful:
\begin{align}
\dvg_\Si \frac{x}{\abs{x}^2} &= \frac{2\abs{x^\perp}^2}{\abs{x}^4} \label{eqn:sob-comp1}\\
\dvg_\Si \frac{x}{\abs{x}} &= \frac{1}{\abs{x}} + \frac{\abs{x^\perp}^2}{\abs{x}^3} \geq \frac{1}{\abs{x}}, \label{eqn:sob-comp2}
\end{align}
where $x=u(p)$ is the position vector of $u$ in $\R^d$.
Throughout, for ease of notation, we will confuse the position vector $x(p)$ (or $y(p)$) with the point $p\in\Si$, even though $u$ is an immersion.

To begin with, we may assume that $f>0$, since otherwise we could replace $f$ with $\abs{f}$.
We may also assume, by translating appropriately, that the immersion maps some interior point of $\Si$ to the origin in $\R^d$.

As promised, we firstly apply \eqref{eqn:sob-firstvar} to $v:=f(x)\frac{x}{\abs{x}^2}$; using \eqref{eqn:sob-comp1} and estimating we obtain
\begin{align}
2\pi f(0) &= -\int_\Si\!\rbrak{\inner{\nabla f,\frac{x}{\abs{x}^2}} + 2f\frac{\abs{x^\perp}^2}{\abs{x}^4} + f\inner{\frac{x^\perp}{\abs{x}^2},H}} + \int_{\pd\Si}\! f\inner{\frac{x,}{\abs{x}^2},\rmn} \nonumber\\
	&\leq \int_\Si\! \rbrak{\frac{\abs{\nabla f(x)}}{\abs{x}} + \frac{f(x)\abs{H(x)}}{\abs{x}}} + \int_{\pd\Si}\! \frac{f(x)}{\abs{x}}. \label{eqn:sob-est1}
\end{align}
On the other hand, we secondly apply \eqref{eqn:sob-firstvar} to $v:=f(y)\frac{y}{\abs{y}}$; using the estimate \eqref{eqn:sob-comp2} and estimating further we obtain
\begin{align}
\int_{\Si}\!\frac{f}{\abs{y}}
	&\leq -\int_\Si\!\rbrak{\inner{\nabla f,\frac{y}{\abs{y}}} + f\inner{\frac{y}{\abs{y}},H}} + \int_{\pd\Si}\!f(y)\inner{\frac{y}{\abs{y}},\rmn} \nonumber\\
	&\leq \int_\Si\!\rbrak{\abs{\nabla f(y)} + f(y)\abs{H(y)}} + \int_{\pd\Si}\! f(y). \label{eqn:sob-est2}
\end{align}
Translating the origin to an arbitrary point $y\in\Si\subset\R^d$, \eqref{eqn:sob-est1} becomes
$$ 2\pi f(y) \leq \int_\Si\!\rbrak{\frac{\abs{\nabla f(x)}}{\abs{x-y}} + \frac{\abs{f(x)H(x)}}{\abs{x-y}}}\,\ed x + \int_{\pd\Si}\!\frac{f(x)}{\abs{x-y}}\,\ed x. $$
Multiplying by $f(y)$ and integrating with respect to $y$ gives
\begin{equation}
2\pi\int_\Si\!f^2 \leq \int_\Si\rbrak{\int_\Si\!\frac{f(y)}{\abs{x-y}}\,\ed y}\rbrak{\abs{\nabla f(x)} + \abs{f(x)H(x)}}\,\ed x
+ \int_{\pd\Si}\!\rbrak{\frac{f(y)}{\abs{x-y}}\,\ed y}f(x)\,\ed x.
\label{eqn:sob-est3}
\end{equation}
On the other hand, translating the origin in \eqref{eqn:sob-est2} to an arbitrary $x$, we have
$$ \int_\Si\!\frac{f(y)}{\abs{x-y}}\,\ed y \leq \int_\Si\!\rbrak{\abs{\nabla f(y)} + \abs{f(y)H(y)}}\,\ed y + \int_{\pd\Si}\!f(y)\,\ed y, $$
which, applied to the above \eqref{eqn:sob-est3}, gives
\begin{align*}
2\pi\int_\Si\! f^2 \leq \sbrak{\int_\Si\!\rbrak{\abs{\nabla f} + \abs{fH}} + \int_{\pd\Si} f}^2.
\end{align*}
Estimating now $\int_{\pd\Si}\!f$ using the method of \cite[Lemma 2.1]{edelen-2016} gives
$$ \int_{\pd\Si}\!\abs{f}
\leq
\frac{\max\abs{X}}{\sin\theta}\int_\Si\!\abs{\nabla f}
+ \frac{2\max\abs{\nabla X}}{\sin\theta}\int_\Si\!\abs{f}
+ \frac{\max\abs{X}}{\sin\theta}\int_\Si\!\abs{fH},
$$
where $X$ is a vector field on $M$ such that on $\pd M$, $X=N$.
Combined with the above, this yields
\begin{equation*}
\sqrt{2\pi}\rbrak{\int_\Si\!f^2}^{1/2}
\leq
\rbrak{1+\frac{\max\abs{X}}{\sin\theta}}\int_\Si\!\abs{\nabla f}
+ \rbrak{1+\frac{\max\abs{X}}{\sin\theta}}\int_\Si\!\abs{fH}
+ \frac{2\max\abs{\nabla X}}{\sin\theta}\int_\Si\!\abs{f}.
\end{equation*}
We may take specifically, for each $\eps\in(0,\rho)$, $X$ to be the vector field $X_\eps$ obtained in Proposition \ref{prop:extend-normal}; we obtain
\begin{equation*}
\sqrt{2\pi}\rbrak{\int_\Si\!f^2}^{1/2}
\leq
\rbrak{1+\frac{1}{\sin\theta}}\int_\Si\!\abs{\nabla f}
+ \rbrak{1+\frac{1}{\sin\theta}}\int_\Si\!\abs{fH}
+ \frac{2}{(\rho-\eps)\sin\theta}\int_\Si\!\abs{f},
\end{equation*}
and taking $\eps\to0$ yields the result.
\end{proof}

\appendix

\section{Variations of the area, enclosed-volume, and wetting functionals}

In a classical work on capillary \cmc{} surfaces, Ros and Souam \cite{ros-souam-1997} derive the second variation formula \eqref{eqn:hess-mod-area-cmc-cap} at a critical point of $\calA^{h,\theta}$ (\ie{} a capillary \cmc{} surface) for volume-preserving and boundary-area-preserving variations.
In our work, however, we require these formulae away from critical points and for non-admissible variations.
For completeness, we give the second variations of the area $\calA$, volume $\calV$, and wetting $\calW$ functionals away from critical points, that is, at boundaried branched immersions.
We combine these functionals to give the second variation for $\calA^{h,\theta}$, where $h$ and $\theta$ are not necessarily constants; \ie{} for the second variations for prescribed mean curvature and contact angle problem.
When the prescribed mean curvature and contact angles are constant, we recover the classical formula in Lemma \ref{lem:hess-mod-area-cmc-cap}, but holding for any variation, not just those which are volume- and boundary-area-preserving.
Moreover, we have the second variation formula for the modified Dirichlet energy $\calE^{h,\theta}$ in Lemma \ref{lem:hess-mod-energy}.

For those readers who have jumped straight to this appendix, let us state that we rely on some terminology and notations set up in Section \ref{sec:prelims}.

\subsection{The area functional}

For clarity, we derive the second variation of the area functional.
We have done this for closed surfaces in \cite{seemungal-sharp-2026}, but here we of course must include rather than discard the boundary integrals.
Let $u:\Si\to M$ be a boundaried branched immersion.
It is convenient to isometrically embed $M\into \R^d$.
Denote by $D$ the Euclidean connection on $\R^d$, by $\nabla$ the Levi-Civita connection on $M$, by $\perp$ the projection in $\R^d$ normal to $\Si$, by $\perp_M$ the projection in $\R^d$ normal to $M$, by $A$ the second fundamental form of $\Si$ in $M$, and by $\sff$ the second fundamental form of $M$ in $\R^d$.

Given a smooth variation $\Phi(t):I\times\Si\to M$ through $u$, there exist $v\in\Ga(u^*TM)$ and $w\in\Ga(u^*T\R^d)$ such that
$$ \Phi(t) = u + tv + \frac{1}{2}t^2 w + O(t^3). $$
According then to the usual variational formula (see \eg{}\cite{simon-1983}), we have
\begin{equation*}
\derivat{^2}{t^2}{t=0}\calA(\Phi(t))
= \geomint{\Si}{
\abs{\rbrak{Dv}^\perp}^2
+ \rbrak{\dvg_\Si v}^2
- \inner{E_i,\nabla_{E_j}v}\inner{E_j,\nabla_{E_i}v}
+ \dvg_\Si w
\,\ed\Si,
}
\end{equation*}
where we suppress summation over the orthonormal frame $E_i$ of $T_p\Si$.
Now we split $D_{E_i}v=\nabla v + \sff(E_i,v)$, and use the Gauss equations to introduce a $\Rm^M$ term.
Moreover, since $w=\Phi''(0)$, we have $w^{\perp_M}=\sff(v,v)$, so we also have the splitting $w=w^{\tgt_M} + \sff(v,v)$, giving
\begin{equation}
\label{eqn:twoderivs-area-initialform}
\derivat{^2}{t^2}{t=0}\calA(\Phi(t))
= \geomint{\Si}{
\abs{\rbrak{\nabla v}^\perp}^2
- \Rm^M(v,E_i,E_i,v)
+ \rbrak{\dvg_\Si v}^2
- \inner{E_i,\nabla_{E_j}v}\inner{E_j,\nabla_{E_i}v}
+ \dvg_\Si w^{\tgt_M}
\,\ed\Si.
}
\end{equation}
Somewhat abusively, we just write
\begin{equation}
\label{eqn:twoderivs-area-untensorial}
\derivat{^2}{t^2}{t=0}\calA(\Phi(t))
= \geomint{\Si}{
\abs{\rbrak{\nabla v}^\perp}^2
- \Rm^M(v,E_i,E_i,v)
+ \rbrak{\dvg_\Si v}^2
- \inner{E_i,\nabla_{E_j}v}\inner{E_j,\nabla_{E_i}v}
+ \dvg_\Si(\nabla_vv)
\,\ed\Si.}
\end{equation}

\begin{lemma}
Let $u:\Si\to M$ be a boundaried branched immersion, and suppose that $\Phi(t):\Si\to M$ is a smooth variation through $u$ with initial velocity $\Phi'(0)=:v\in\Ga(u^*TM)$, which splits into $v=s+\si$ with $s\in\nu\Si$ and $\si\in\xi$.
Then
\begin{multline}
\label{eqn:twoderivs-area}
\derivat{^2}{t^2}{t=0}\calA(\Phi(t))
=
\geomint{\Si}{
\inner{-\Delta^\perp s,s}
- \Ric^M(s,s)
+ \inner{s,H}^2
- \inner{s,A(E_i,E_j)}^2
- \inner{\nabla_ss,H}
+ \inner{\sbrak{\si,s},H}
\,\ed\Si
}
\\
+ \geomint{\pd\Si}{
\inner{\nabla_\rmn s,s}
+ \dvg_\Si\si\inner{\si,\rmn}
- 2\inner{s,H}\inner{\si,\rmn}
+ \inner{\nabla_ss,\rmn}
- \inner{\nabla_\si s,\rmn}
+ \inner{\nabla_s\si,\rmn}
}
\,\ed\tau
\end{multline}
where $E_i$ is a local orthonormal frame for $\xi$, $A$ is the second fundamental form of $u:\Si\to M$, and $H=\tr A$ is the mean curvature.
\end{lemma}
\begin{proof}
Beginning with \eqref{eqn:twoderivs-area-untensorial}, we use the fact that $\dvg_\Si X= \dvg_\Si\rbrak{X^\tgt} - \inner{X,H}$ (for $X\in\Ga(u^*TM)$) and split $v=s+\si$ to obtain
\begin{multline*}
\derivat{^2}{t^2}{t=0}\calA(\Phi(t)) =
\int_\Si\!
\abs{\rbrak{\nabla s}^\perp}^2
+ \abs{\rbrak{\nabla_{E_i}\si}^\perp}^2
+ 2\inner{\nabla_{E_i}^\perp s,\rbrak{\nabla_{E_i}\si}^\perp}
- \Rm^M\rbrak{s,E_i,E_i,s} \\
- 2\Rm^M(s,E_i,E_i,\si)
- \Rm^M(\si,E_i,E_i,\si)
+ \rbrak{\dvg_\Si\si}^2 \\
+ \inner{s,H}^2
- 2\inner{s,H}\dvg_\Si\si
- \inner{\nabla_{E_i}s,E_j}\inner{E_i,\nabla_{E_j}s} \\
- 2\inner{\nabla_{E_i}s,E_j}\inner{E_i,\nabla_{E_j}\si}
- \inner{\nabla_{E_i}\si,E_j}\inner{E_i,\nabla_{E_j}\si} \\
- \inner{\nabla_ss,H}
- \inner{\nabla_s\si,H}
- \inner{\nabla_\si s,H}
- \inner{A(\si,\si),H} \\
+ \dvg_\Si\rbrak{\nabla_ss + \nabla_\si s + \nabla_s\si + \nabla_\si\si}^\tgt.
\end{multline*}
Above and throughout the following, we suppress the notation for summing over $E_i$.
Now, we use the following easily-derived formulae,
\begin{align*}
\rbrak{\nabla_{E_i}\si}^\perp &= A(E_i,\si) &
\Rm^M(s,E_i,E_i,s) &= \Ric^M(s,s) \\
\inner{\nabla_{E_i}s,E_j}\inner{E_i,\nabla_{E_j}s} &= \inner{s,A(E_i,E_j)}^2 &
\inner{\nabla_{E_i}s,E_j}\inner{E_i,\nabla_{E_j}\si}&=-\inner{s,A\rbrak{\nabla^\tgt_{E_j}\si,E_j}}
\end{align*}
in the above, giving
\begin{multline*}
\derivat{^2}{t^2}{t=0}\calA(\Phi(t))
= \int_\Si\!
\abs{\nabla^\perp s}^2
+ \abs{A(E_i,\si)}^2
+ 2\inner{\nabla^\perp_{E_i} s,A(E_i,\si)}
- \Ric^M(s,s)
- 2\Rm^M(s,E_i,E_i,\si) \\
- \Rm^M(\si,E_i,E_i,\si)
+ \rbrak{\dvg_\Si\si}^2
+ \inner{s,H}^2
- 2\inner{s,H}\dvg_\Si\si \\
- \inner{s,A(E_i,E_j)}^2
+ 2\inner{s,A\rbrak{\nabla^\tgt_{E_i}\si,E_i}}
- \inner{\nabla_{E_i}\si,E_j}\inner{E_i,\nabla_{E_j}\si} \\
- \inner{\nabla_ss,H}
- \inner{\nabla_s\si,H}
- \inner{\nabla_\si s,H}
- \inner{A(\si,\si),H} \\
+ \dvg_\Si\rbrak{\nabla_ss + \nabla_\si s + \nabla_s\si + \nabla_\si\si}^\tgt
\,\ed\Si.
\end{multline*}
Separating the formula into the $s$-parts, the $\si$-parts, the mixed-$s$-$\si$-parts, and the boundary-terms, we arrive at
\begin{multline*}
\derivat{^2}{t^2}{t=0}\calA(\Phi(t))
= \int_\Si\!
\abs{\nabla^\perp s}^2
- \Ric^M(s,s)
+ \inner{s,H}^2
- \inner{s,A(E_i,E_j)}^2
- \inner{\nabla_ss,H} \\
\begin{array}{ll}
\left.\begin{array}{ll}
+ \abs{A(E_i,\si)}^2
- \Rm^M(\si,E_i,E_i,\si)
+ \rbrak{\dvg_\Si\si}^2 \\
- \inner{\nabla_{E_i}\si,E_j}\inner{E_i,\nabla_{E_j}\si}
- \inner{A(\si,\si),H}
\end{array}\right\}:=\calB
\\
\\
\left.\begin{array}{ll}
+ 2\inner{\nabla^\perp_{E_i} s,A(E_i,\si)}
- 2\Rm^M(s,E_i,E_i,\si)
- 2\inner{s,H}\dvg_\Si\si \\
+ 2\inner{s,A\rbrak{\nabla^\tgt_{E_i}\si,E_i}}
- \inner{\nabla_s\si,H}
- \inner{\nabla_\si s,H}
\end{array}\right\}:=\calC
\end{array} \\
+ \dvg_\Si\rbrak{\nabla_ss + \nabla_\si s + \nabla_s\si + \nabla_\si\si}^\tgt
\,\ed\Si.
\end{multline*}
We have the following two claims which, when combined with the above, give the conclusion of the lemma.

\textbf{Claim I.}
$$ \calB = \dvg_\Si\rbrak{\si\dvg_\Si\si - \nabla^\tgt_\si\si}. $$

\textbf{Claim II.}
$$ \calC = \dvg_\Si\rbrak{-2\rbrak{\nabla_\si s}^\tgt - 2\inner{s,H}\si} - \inner{\nabla_s\si, H} + \inner{\nabla^\perp_\si s, H} $$

\textit{Proof of Claim I.} On the one hand, the Gauss equations give $\Rm^M(\si,E_i,E_i,\si)=\Ric^\Si(\si,\si) - \inner{A(\si,\si),H} + \abs{A(E_i,\si)}^2$.
On the other hand, we have the following computation
\begin{align*}
\dvg_\Si\rbrak{\nabla^\tgt_\si\si}
&= \nabla_{E_i}\inner{\nabla_\si\si,E_i} \\
&= \nabla_{E_i}\rbrak{\inner{\si,E_j}\inner{\nabla^\tgt_{E_j}\si,E_i}} \\
&= \inner{\nabla^\tgt_{E_i}\si,E_j}\inner{\nabla^\tgt_{E_j}\si,E_i}
	+ \inner{\si,E_j}\inner{\nabla^\tgt_{E_i}\nabla^\tgt_{E_j}\si,E_i} \\
&= \inner{\si,E_j}\inner{\nabla^\tgt_{E_j}\nabla^\tgt_{E_i}\si,E_i}
	+ \inner{\si,E_j}\Rm^\Si(E_i,E_j,\si,E_i)
	+ \inner{\nabla^\tgt_{E_i}\si,E_j}\inner{\nabla^\tgt_{E_j}\si,E_i} \\
&= \nabla_\si \inner{\nabla^\tgt_{E_i}\si,E_i}
	+ \Rm^\Si(E_i,\si,\si,E_i)
	+ \inner{\nabla^\tgt_{E_i}\si,E_j}\inner{\nabla^\tgt_{E_j}\si,E_i} \\
&= \nabla_\si\dvg_\Si\si
	+ \Ric^\Si(\si,\si)
	+ \inner{\nabla_{E_i}\si,E_j}\inner{E_i,\nabla_{E_j}\si}.
\end{align*}
Combining these two hands, we have
$$ \calB = \nabla_\si\rbrak{\dvg_\Si\si} + \rbrak{\dvg_\Si\si}^2 - \dvg_\Si\rbrak{\nabla^\tgt_\si\si}; $$
applying the formula $\dvg_\Si\rbrak{\si\dvg_\Si\si}=\rbrak{\dvg_\Si\si}^2 + \nabla_\si\rbrak{\dvg_\Si\si}$ yields the claim.

\textit{Proof of Claim II.}
Starting with the Codazzi--Mainardi equation, we proceed with the following computation
\begin{align*}
\Rm^M(s,E_i,E_i,\si) &= \Rm^M(\si,E_i,E_i,s) \\
&= \inner{\rbrak{\nabla_\si A}(E_i,E_i),s}
	- \inner{\rbrak{\nabla_{E_i}A}(\si,E_i),s} \\
&= \inner{\nabla_\si H,s}
	- 2\inner{A\rbrak{\nabla^\tgt_\si E_i,E_i},s}
	- \inner{\nabla_{E_i}A(\si,E_i),s}
	+ \inner{A\rbrak{\nabla^\tgt_{E_i}\si,E_i},s} \\
&= \nabla_\si\inner{H,s}
	- \inner{H,\nabla^\perp_\si s}
	- 2\inner{\si,E_j}\inner{A\rbrak{\nabla^\tgt_{E_j}E_i,E_i},s}
	\\&\qquad
	- \nabla_{E_i}\inner{A(\si,E_i),s}
	+ \inner{A(\si,E_i),\nabla_{E_i}s}
	+ \inner{A\rbrak{\nabla^\tgt_{E_i}\si,E_i},s} \\
&= \nabla_\si\inner{H,s}
	- \inner{H,\nabla^\perp_\si s}
	+ \nabla_{E_i}\inner{E_i,\nabla_\si s}
	+ \inner{A(\si,E_i),\nabla_{E_i}s}
	+ \inner{A\rbrak{\nabla^\tgt_{E_i}\si,E_i},s};
\end{align*}
of course, the third term after the fourth equality vanishes.
Combining this with the expression for $\calC$ above, we have after some cancellations
$$ \calC = - 2\nabla_{E_i}\inner{E_i, \nabla_\si s}
	- 2\inner{s,H}\dvg_\Si\si
	- 2\nabla_\si\inner{H,s}
	- \inner{\nabla_s\si,H}
	+ \inner{\nabla^\perp_\si s,H}.
$$
Expanding the claimed expression for $\calC$ reveals that it is equal to the above, which gives the claim, and hence the lemma.
\end{proof}

\begin{lemma}
Let $u:\Si\to M$ be a boundaried branched immersion, and suppose that $v\in\Ga_{\pd M}(u^*TM)$, which splits into $v=s+\si$ with $s\in\nu\Si$ and $\si\in\xi$.
Then
\begin{equation}
\firstvar{\calA}{u}{v}
=
-\geomint{\Si}{\inner{s,H}\,\ed\Si}
+\geomint{\pd\Si}{\inner{\si,\rmn}\,\ed\tau}
\end{equation}
and
\begin{multline}
\label{eqn:secondvar-area}
\secondvar{\calA}{u}{v,v}
=
\geomint{\Si}{
\inner{-\Delta^\perp s,s}
- \Ric^M(s,s)
+ \inner{s,H}^2
+ \inner{A(\si,\si),H}
- \inner{s,A(E_i,E_j)}^2
+ 2\inner{\nabla_\si s, H}
\,\ed\Si
}
\\
+ \geomint{\pd\Si}{
\inner{\nabla_\rmn s,s}
+ \dvg_\Si\si\inner{\si,\rmn}
- 2\inner{s,H}\inner{\si,\rmn}
- 2\inner{\nabla_\si s,\rmn}
- \inner{\nabla_\si\si,\rmn}
}
\,\ed\tau
\end{multline}
where $E_i$ is a local orthonormal frame for $\xi$, $A$ is the second fundamental form of $u:\Si\to M$, and $H=\tr A$ is the mean curvature.
\end{lemma}
\subsection{The enclosed-volume functional}

Recall that, if $\Phi:I\times\Si\to M$ is a variation starting at a boundaried branched immersion $u:\Si\to M$, then the generalised volume of this variation at time $t\in I$ is precisely
\begin{equation}
\calV(\Phi(t)) = \geomint{[0,t]\times\Si}{\Phi^*\om},
\end{equation}
where of course $\om=h\Vol^M$ for some prescribed mean curvature $h\in\smooth{M}$.
In the paper, we of course take $h$ to be constant, but we may as well be more general here since it does not require too much effort.

\begin{lemma}
\label{lem:pullback-derivs}
Let $u:\Si\to M$ be a boundaried branched immersion, and suppose that $\Phi:I\times\Si\to M$ is a variation of $u$.
Then we have
\begin{align}
\calV'(t) &=
	\geomint{\Si}{\Phi(t)^*\iota_\psi\om}
\label{firstderiv-vol-gen}\\
\calV''(t) &=
	\geomint{\Si}{\Phi(t)^*\rbrak{\iota_\psi\ed\iota_\psi\om}}.
\end{align}
\end{lemma}
\begin{proof}
We begin with the useful formula
\begin{equation}
\Phi^*\om = \Phi(s)^*\om + \ed t\wedge\Phi(s)^*\iota_\psi\om.
\label{derivs-vol-gen-useful}
\end{equation}
Indeed, we have
$$
V(t)
= \geomint{[0,t]\times\Si}{\Phi^*\om}
= \geomint{[0,t]\times\Si}{\Phi(s)^*\om + \ed s\wedge\Phi(s)^*\iota_\psi\om}
= \integ{0}{t}{\rbrak{
	\geomint{\Si}{\Phi(s)^*\iota_\psi\om}
	}}{s},
$$
allowing us to differentiate and immediately obtain the first derivative
$$
\calV'(t)
= \geomint{\Si}{\Phi(t)^*\iota_\psi\om}.
$$
For the second derivative, notice first that (again by the useful formula \eqref{derivs-vol-gen-useful}) we have
$$
\calV'(t)
= \geomint{\Si}{\Phi^*\iota_\psi\om - \ed s\wedge\Phi(s)^*\iota_\psi\iota_\psi\om}
= \geomint{\Si}{\Phi^*\iota_\psi\om},
$$
since $\iota_\psi\iota_\psi=0$.
This allows us to use Cartan's magic formula $\calL_{\pd_t}=\iota_{\pd_t}\ed + \ed\iota_{\pd_t}$, and then calling again upon \eqref{derivs-vol-gen-useful} to switch back to pulling back to $\Si$ gives
\begin{align*}
\calV''(t)
= \geomint{\Si}{\Phi^*\iota_\psi\ed\iota_\psi\om
	+ \ed\iota_\psi\iota_\psi\om}
= \geomint{\Si}{\Phi(t)^*\iota_\psi\ed\iota_\psi\om
	+ \ed t\wedge\Phi(t)^*\iota_\psi\iota_\psi\ed\iota_\psi\om}
= \geomint{\Si}{\Phi(t)^*\iota_\psi\ed\iota_\psi\om},
\end{align*}
as desired.
\end{proof}

\begin{lemma}
Let $u:\Si\to M$ be a boundaried branched immersion.
Then for any $v\in\Ga_{\pd M}(u^*TM)$, we have
\begin{align}
\firstvar{\calV}{u}{v}
&=
\geomint{\Si}{u^*\iota_v\om}
=
\geomint{\Si}{\inner{v,h\nu}\,\ed\Si}
\label{firstvar-vol-gen}\\
\secondvar{\calV}{u}{v,v}
&=
\geomint{\Si}{u^*\rbrak{\iota_v\ed\iota_v\om - \iota_{\nabla_vv}\om}}
\nonumber\\
&=
\geomint{\Si}{
\rbrak{
	\rbrak{\nabla_v h}\Vol(v,u_x,u_y)
	+ h\Vol(v,\nabla_{u_x}v,u_y)
	+ h\Vol(v,u_x,\nabla_{u_y}v)
}
\,\ed x\wedge\ed y
}
\label{hess-vol-gen}\\
&=
\begin{multlined}[t]
\geomint{\Si}{
\rbrak{\nabla_\nu h}
\abs{s}^2
-2\inner{\nabla_\si s,h\nu}
-\inner{A(\si,\si),h\nu}
-\inner{s,h\nu}\inner{s,H}
\,\ed\Si}
\\
+ \geomint{\pd\Si}{\inner{s,h\nu}\inner{\si,\rmn}\,\ed\tau}.
\label{hess-vol-gen-detail}
\end{multlined}
\end{align}
\end{lemma}
\begin{proof}
For the first variation, we simply take the above formula \eqref{firstderiv-vol-gen} and evaluate at $t=0$.
Then, in coordinates $x,y$ of $\Si$ where $\Vol(\nu,u_x,u_y)=e^{2\la}$, we have
$$ u^*(\iota_v\om)
= h\Vol(v,u_x,u_y)\,\ed x\wedge\ed y
= \inner{v,h\nu}e^{2\la}\,\ed x\wedge\ed y. $$
The formula follows.

For the second variation, we have by definition that
$$ \secondvar{\calV}{u}{v,v}
=
\calV''(0) - \firstvar{\calV}{u}{\nabla_vv}
= \geomint{\Si}{u^*\rbrak{\iota_v\ed\iota_v\om - \iota_{\nabla_vv}\om}}.
$$
Now a simple computation shows that
$$ u^*\iota_v\ed\iota_v\om - \iota_{\nabla_v}\om
=
\rbrak{
\rbrak{\nabla_v\om}(v,u_x,u_y)
+ \om\rbrak{v,\nabla_{u_x}v,u_y}
- \om\rbrak{v,\nabla_{u_y}v,u_x}
}
\,\ed x\wedge \ed y
$$
Using the fact that $\om=h\Vol^M$ and that $\nabla\Vol^M=0$ yeilds the second formula for the second variation.

For the third formula, notice that $\Vol^M\rbrak{v,u_x,u_y}=e^{2\la}\inner{s,\nu}$, so
\begin{equation}
\rbrak{\nabla_v h}\Vol^M\rbrak{v,u_x,u_y}=\rbrak{\nabla_\si h + \nabla_s h}e^{2\la}\inner{s,\nu}.
\label{hess-vol-gen-detail-1}
\end{equation}
Furthermore,
\begin{align*}
\Vol^M\rbrak{v,\nabla_{u_x}v,u_y}
&= \inner{s,\nu}\Vol^M\rbrak{\nu,\nabla_{u_x}v,u_y}
	+ \Vol^M\rbrak{\si,\nabla_{u_x}v,u_y} \\
&= \inner{s,\nu}\inner{\nabla_{u_x}v,u_x}\Vol^M\rbrak{\nu,e^{-2\la}u_x,u_y}
	+ \inner{\nabla_{u_x}v,\nu} \Vol^M\rbrak{\si,\nu,u_y} \\
&= \inner{s,\nu}\inner{\nabla_{u_x}v,u_x}
	- \inner{\nabla_{u_x}v,\nu}\inner{\si,u_x}.
\end{align*}
Similarly,
$$ \Vol^M\rbrak{v,u_x,\nabla_{u_y}v} =
\inner{s,\nu}\inner{\nabla_{u_y}v,u_y}
- \inner{\nabla_{u_y}v,\nu}\inner{\si,u_y}. $$
Therefore,
\begin{align*}
h\Vol^M\rbrak{v,\nabla_{u_x}v,u_y} + h\Vol^M\rbrak{v,u_x,\nabla_{u_y}v}
&= e^{2\la}\inner{s,h\nu}\dvg_\Si v
	- e^{2\la}\inner{\nabla_\si v,h\nu} \nonumber\\
&= e^{2\la}\inner{s,h\nu}\dvg_\Si\si
	- e^{2\la}\inner{s,h\nu}\inner{s,H} \nonumber\\
&\qquad - e^{2\la}\inner{\nabla_\si s,h\nu}
	- e^{2\la}\inner{\nabla_\si\si,h\nu}.
\end{align*}
Now, notice that
\begin{align*}
\dvg_\Si\rbrak{\inner{s,h\nu}\si}
&= \nabla_\si \inner{s,h\nu}
	+ \inner{s,h\nu}\dvg_\Si\si \\
&= \inner{\nabla_\si s,h\nu}
	+ \rbrak{\nabla_\si h}\inner{s,\nu}
	+ h\inner{s,\nabla^\perp_\si\nu}
	+ \inner{s,h\nu}\dvg_\Si\si \\
&= \inner{\nabla_\si s,h\nu}
	+ \rbrak{\nabla_\si h}\inner{s,\nu}
	+ \inner{s,h\nu}\dvg_\Si\si,
\end{align*}
where the last equality follows because $h\inner{s,\nabla^\perp_\si\nu} = h\inner{s,\nu}\inner{\nu,\nabla^\perp_\si\nu} = h\inner{s,\nu}\nabla_\si\inner{\nu,\nu} = 0$.
Therefore,
\begin{multline}
e^{-2\la}\rbrak{h\Vol^M\rbrak{v,\nabla_{u_x}v,u_y}
	+ h\Vol^M\rbrak{v,u_x,\nabla_{u_y}v}} \\
= -\rbrak{\nabla_\si h}\inner{s,\nu}
	- 2\inner{\nabla_\si s,h\nu}
	- \inner{\nabla_\si\si,h\nu}
	-\inner{s,h\nu}\inner{s,H}
	+\dvg_\Si\rbrak{\inner{s,h\nu}\si}.
\label{hess-vol-gen-detail-2}
\end{multline}
Combining \eqref{hess-vol-gen-detail-1} with \eqref{hess-vol-gen-detail-2} yields the formula.
\end{proof}

\subsection{The wetting functional}

The wetting functional is precisely, for $\Phi:I\times\Si\to M$ as above,
$$ W(\Phi) = \geomint{I\times\pd\Si}{\Phi^*\chi}, $$
where $\chi=\cos\theta\Area^{\pd M}\in\Om^2(\pd M)$.
We refer back to our computations in Lemma \ref{lem:pullback-derivs}.
Note that we do not assume $\theta$ to be constant, only that $\theta\in\smooth{\pd M,\R}$, so that we situate ourselves rather generally.

\begin{lemma}
Let $u:\Si\to M$ be a boundaried branched immersion, and suppose that $\Phi:I\times\Si\to M$ is a variation of $u$.
Then for any $v\in\Ga_{\pd M}(u^*TM)$, we have
\begin{align}
\firstvar{\calW}{u}{v}
&= \geomint{\pd\Si}{u^*\iota_v\chi}
= - \geomint{\pd\Si}{\cos\theta\inner{v,\hat{\nu}}\,\ed\tau}
\label{firstvar-wet-general}\\
\secondvar{\calW}{u}{v,v}
&= \geomint{\pd\Si}{u^*\rbrak{\iota_v\ed\iota_v\chi - \iota_{\nabla_vv}\chi}}
\nonumber\\
&= \geomint{\pd\Si}{
	\rbrak{\nabla_v\cos\theta}\Area^{\pd M}(v,\dot{\ga})
	+ \cos\theta\Area^{\pd M}(v,\nabla_{\dot{\ga}}v)
	\,\ed\tau
}
\nonumber\\
&= \geomint{\pd\Si}{
	- \rbrak{\nabla_v\cos\theta}\inner{v,\hat{\nu}}
	+ \cos\theta\inner{v,\dot{\ga}}\inner{\nabla_{\dot{\ga}}v,\hat{\nu}}
	- \cos\theta\inner{v,\hat{\nu}}\inner{\nabla_{\dot{\ga}}v,\dot{\ga}}
	\,\ed\tau
}
\nonumber\\
&=
\begin{multlined}[t]
\geomint{\pd\Si}{
	- \rbrak{\nabla_v\cos\theta}\inner{v,\hat{\nu}}
- \pd_\tau\rbrak{\frac{\cos\theta}{\cos\alpha}}\inner{\si,\gadot}\inner{\si,\rmn}
\\
-\frac{\cos\theta\sin\alpha}{\cos\alpha}\inner{A^{\pd M}(\gadot,\gadot),N}\abs{v}^2\,\ed\tau,
}
\label{hess-wet-general}
\end{multlined}
\end{align}
where $\cos\alpha = \inner{\hat{\nu},\rmn}$ as in fig.~\ref{fig:normals}.
\end{lemma}
\begin{proof}
Formula \eqref{firstvar-wet-general} and the first three second variation formulae follow from an argument identical that of Lemma \ref{lem:pullback-derivs}.
So we focus on the last formula \eqref{hess-wet-general}.
We begin, introducing $X$ as notation for the primary object of our attention, by decomposing $\cos\alpha\hat{\nu}=\rmn-\sin\alpha N$ and $v=s+\si$,
\begin{align}
X &= \cos\theta\Area^{\pd M}\rbrak{v,\nabla_{\dot{\ga}}v}
\nonumber\\
&
= \cos\theta\inner{v,\dot{\ga}}\inner{\nabla_{\dot{\ga}}v,\hat{\nu}}
- \cos\theta\inner{v,\hat{\nu}}\inner{\nabla_{\dot{\ga}}v, \dot{\ga}}
\nonumber\\
&
=
\frac{\cos\theta}{\cos\alpha}\inner{\si,\dot{\ga}}\inner{\nabla_{\dot{\ga}}v,\rmn}
- \frac{\cos\theta\sin\alpha}{\cos\alpha}\inner{\si,\dot{\ga}}\inner{\nabla_{\dot{\ga}}v,N}
- \frac{\cos\theta}{\cos\alpha}\inner{\si,\rmn}\inner{\nabla_{\dot{\ga}}v,\dot{\ga}}
\nonumber\\
&
=
\frac{\cos\theta}{\cos\alpha}\inner{\si,\dot{\ga}}\inner{\nabla_{\dot{\ga}}\si,\rmn}
- \frac{\cos\theta}{\cos\alpha}\inner{\si,\rmn}\inner{\nabla_{\dot{\ga}}\si,\dot{\ga}}
- \frac{\cos\theta}{\cos\alpha}\inner{\si,\dot{\ga}}\inner{A(\dot{\ga},\rmn),s}
\nonumber\\&\qquad
+ \frac{\cos\theta}{\cos\alpha}\inner{\si,\rmn}\inner{A(\dot{\ga},\dot{\ga}),s}
- \frac{\cos\theta\sin\alpha}{\cos\alpha}\inner{\si,\dot{\ga}}\inner{A^{\pd M}(\dot{\ga},v),N}.
\label{hess-wet-x}
\end{align}
The second term vanishes, because
\begin{align}
\inner{\nabla_{\dot{\ga}}\si,\dot{\ga}}
&=
\inner{\sbrak{\dot{\ga},\si},\dot{\ga}}
+ \inner{\nabla_\si\dot{\ga},\dot{\ga}}
\nonumber\\
&= \inner{\si,\rmn}\inner{\sbrak{\dot{\ga},\rmn},\dot{\ga}}
\nonumber\\
&= 0,
\label{hess-wet-x-2}
\end{align}
where the last equality follows from our choice of coordinates.
We can now simplify the first term, much of which is integrated out by Stokes' theorem;
indeed, observe that,
\begin{align}
\pd_\tau\rbrak{\frac{\cos\theta}{\cos\alpha}\inner{\si,\dot{\ga}}\inner{\si,\rmn}}
&=
\pd_\tau\rbrak{\frac{\cos\theta}{\cos\alpha}}\inner{\si,\dot{\ga}}\inner{\si,\rmn}
+ \frac{\cos\theta}{\cos\alpha}\inner{\nabla_{\dot{\ga}}\si,\dot{\ga}}\inner{\si,\rmn}
+ \frac{\cos\theta}{\cos\alpha}\inner{\si,\dot{\ga}}\inner{\nabla_{\dot{\ga}}\si,\rmn}
\nonumber
\\
&=
\pd_\tau\rbrak{\frac{\cos\theta}{\cos\alpha}}\inner{\si,\dot{\ga}}\inner{\si,\rmn}
+ \frac{\cos\theta}{\cos\alpha}\inner{\si,\dot{\ga}}\inner{\nabla_{\dot{\ga}}\si,\rmn},
\label{hess-wet-x-1}
\end{align}
of which the latter term is precisely the first term in \eqref{hess-wet-x}.
For the third term, we utilise firstly the fact that $\rmn = \cot\alpha\nu + \frac{1}{\sin\alpha} N$ and secondly that $\nu = \sin\alpha\hat{\nu}-\cos\alpha N$:
\begin{align}
-\frac{\cos\theta}{\cos\alpha}\inner{\si,\dot{\ga}}\inner{A(\dot{\ga},\rmn),s}
&=
-\frac{\cos\theta}{\cos\alpha}\inner{\si,\dot{\ga}}\inner{s,\nu}\inner{\nabla_{\gadot}\rmn,\nu}
\nonumber\\
&=
-\frac{\cos\theta}{\sin\alpha}\inner{\si,\gadot}\inner{s,\nu}\inner{\nabla_{\gadot}\nu,\nu}
- \frac{\cos\theta}{\cos\alpha\sin\alpha}\inner{\si,\gadot}\inner{s,\nu}\inner{\nabla_{\gadot}N,\nu}
\nonumber\\
&=
-\frac{\cos\theta}{\cos\alpha}\inner{\si,\gadot}\inner{s,\nu}\inner{\nabla_{\gadot}N,\hat{\nu}}
+ \frac{\cos\theta}{\sin\alpha}\inner{\si,\gadot}\inner{s,\nu}\inner{\nabla_{\gadot}N,N}
\nonumber\\
&=
\frac{\cos{\theta}}{\cos{\alpha}}\inner{\si,\gadot}\inner{s,\nu}\inner{A^{\pd M}(\dot{\ga},\hat{\nu}),N}.
\label{hess-wet-x-3}
\end{align}
For the fourth term, we use the fact that $\nu = -\frac{1}{\cos\alpha} N + \tan\alpha\rmn$:
\begin{align}
\frac{\cos\theta}{\cos\alpha}\inner{\si,\rmn}\inner{A(\gadot,\gadot),s}
&=
\frac{\cos\theta}{\sin\alpha}\inner{s,\nu}^2\inner{A(\gadot,\gadot),\nu}
\nonumber\\
&=
-\frac{\cos\theta}{\sin\alpha\cos\alpha}\inner{\nabla_{\gadot}\gadot,N}\abs{s}^2
+ \frac{\cos\theta}{\cos\alpha}\inner{\nabla_{\gadot}\gadot,\rmn}\abs{s}^2
\nonumber\\
&=
-\frac{\cos\theta}{\sin\alpha\cos\alpha}\inner{A^{\pd M}(\gadot,\gadot),N}\abs{s}^2.
\label{hess-wet-x-4}
\end{align}
Finally for the fifth term, we simply expand
\begin{align}
-\frac{\cos\theta\sin\alpha}{\cos\alpha}\inner{\si,\gadot}\inner{A^{\pd M}(\gadot,v),N}
&=
-\frac{\cos\theta\sin\alpha}{\cos\alpha}\inner{\si,\gadot}^2\inner{A^{\pd M}(\gadot,\gadot),N}
\nonumber\\&\qquad
-\frac{\cos\theta\sin\alpha}{\cos\alpha}\inner{\si,\gadot}\inner{v,\hat{\nu}}\inner{A^{\pd M}(\gadot,\hat{\nu}),N}
\nonumber\\
&=
-\frac{\cos\theta\sin\alpha}{\cos\alpha}\inner{\si,\gadot}^2\inner{A^{\pd M}(\gadot,\gadot),N}
\nonumber\\&\qquad
-\frac{\cos\theta}{\cos\alpha}\inner{\si,\gadot}\inner{s,\nu}\inner{A^{\pd M}(\gadot,\hat{\nu}),N}.
\label{hess-wet-x-5}
\end{align}
Utilising \eqref{hess-wet-x-2}, \eqref{hess-wet-x-1}, \eqref{hess-wet-x-3}, \eqref{hess-wet-x-4}, and \eqref{hess-wet-x-5} into \eqref{hess-wet-x}
\begin{align}
X &=
\pd_\tau\rbrak{\frac{\cos\theta}{\cos\alpha}\inner{\si,\gadot}\inner{\si,\rmn}}
- \pd_\tau\rbrak{\frac{\cos\theta}{\cos\alpha}}\inner{\si,\gadot}\inner{\si,\rmn}
\nonumber\\&\qquad
- \frac{\cos\theta}{\sin\alpha\cos\alpha}\inner{A^{\pd M}(\gadot,\gadot),N}\abs{s}^2
- \frac{\cos\theta\sin\alpha}{\cos\alpha}\inner{\si,\gadot}^2\inner{A^{\pd M}(\gadot,\gadot),N}
\nonumber\\
&=
\pd_\tau\rbrak{\frac{\cos\theta}{\cos\alpha}\inner{\si,\gadot}\inner{\si,\rmn}}
- \pd_\tau\rbrak{\frac{\cos\theta}{\cos\alpha}}\inner{\si,\gadot}\inner{\si,\rmn}
-\frac{\cos\theta\sin\alpha}{\cos\alpha}\inner{A^{\pd M}(\gadot,\gadot),N}\abs{v}^2,
\nonumber
\end{align}
where the second equality follows because $\abs{v}^2 = \inner{v,\hat{\nu}}^2+\inner{v,\dot{\ga}}^2 = \frac{1}{\sin^2\alpha}\abs{s}^2 + \inner{v,\dot{\ga}}^2$.
\end{proof}

\subsection{The combined functionals}

Combining the three functionals, we get the following.

\begin{lemma}
\label{lem:vars-combined-general}
Let $u:\Si\to M$ be a boundaried branched immersion.
Then for any $v\in\Ga_{\pd M}(u^*TM)$, we have
\begin{align*}
\firstvar{\calA^{h,\theta}}{u}{v}
&=
\geomint{\Si}{
	\inner{s,h\nu - H}
\,\ed\Si}
+\geomint{\pd\Si}{
	\inner{\si - \frac{\cos\theta}{\cos\alpha}\si,\rmn}
\,\ed\tau},
\end{align*}
where $\cos\alpha=\inner{\rmn,\hat{\nu}}$.
Furthermore, we have
\begin{align}
\secondvar{\calA^{h,\theta}}{u}{v,v}
&=
\begin{multlined}[t]
\geomint{\Si}{
\inner{-\Delta^\perp s,s}
- \Ric^M(s,s)
+ \inner{s,H}\inner{s,H-h\nu}
+ \inner{A(\si,\si),H-h\nu}
\\
- \inner{s,A(E_i,E_j)}^2
+ 2\inner{\nabla_\si s, H-h\nu}
+ \rbrak{\nabla_\nu h} \abs{s}^2
\,\ed\Si}
\end{multlined}
\nonumber\\&\qquad
\begin{multlined}[b]
+ \geomint{\pd\Si}{
	\inner{\nabla_\rmn s,s}
	+ \dvg_\Si\si\inner{\si,\rmn}
	+ \inner{s,h\nu - 2H}\inner{\si,\rmn}
	- 2\inner{\nabla_\si s,\rmn}
	- \inner{\nabla_\si\si,\rmn}
\\
	- \rbrak{\nabla_v\cos\theta}\inner{v,\hat{\nu}}
	+ \cos\theta\inner{v,\dot{\ga}}\inner{\nabla_{\dot{\ga}}v,\hat{\nu}}
	- \cos\theta\inner{v,\hat{\nu}}\inner{\nabla_{\dot{\ga}}v,\dot{\ga}}
\,\ed\tau}.
\end{multlined}
\label{hess-combined-general}
\end{align}
\end{lemma}

Of course, it follows from this first variation formula that $\ed\calA^{h,\theta}(u)\equiv0$ if and only if $u$ is $h\nu = H$ and $\cos\alpha=\cos\theta$.

\begin{lemma}
\label{lem:hess-mod-area-cmc-cap}
Suppose that $h$ and $\theta$ are constants, and that $\ed\calA^{h,\theta}(u)\equiv0$ on $\Ga_{\pd M}(u^*TM)$, \ie{}$u:\Si\to M$ is an $h$-\cmc{} $\theta$-capillary surface.
Then for any $v\in\Ga_{\pd M}(u^*TM)$, we have
\begin{multline}
\secondvar{\calA^{h,\theta}}{u}{v,v}
=
\geomint{\Si}{
\inner{-\Delta^\perp s,s}
- \Ric^M(s,s)
- \inner{s,A(E_i,E_j)}^2
\,\ed\Si}
\\
+ \geomint{\pd\Si}{
	\inner{\nabla_\rmn s,s}
	+ \rbrak{\cot\theta\inner{A(\rmn,\rmn),\nu}
		+ \frac{1}{\sin\theta}\inner{A^{\pd M}(\hat{\nu},\hat{\nu}),N}
	}\abs{s}^2
\,\ed\tau}.
\label{eqn:hess-mod-area-cmc-cap}
\end{multline}
\end{lemma}
\begin{proof}
Beginning with \eqref{hess-combined-general}, we take $h\nu=H$ and $\cos\theta=\cos\alpha$, giving
\begin{align*}
\secondvar{\calA^{h,\theta}}{u}{v,v}
&=
\geomint{\Si}{
\inner{-\Delta^\perp s,s}
- \Ric^M(s,s)
- \inner{s,A(E_i,E_j)}^2
\,\ed\Si}
\\&\qquad
\begin{multlined}[t]
+ \geomint{\pd\Si}{
	\inner{\nabla_\rmn s,s}
	+ \dvg_\Si\si\inner{\si,\rmn}
	- \inner{s,H}\inner{\si,\rmn}
	- 2\inner{\nabla_\si s,\rmn}
	- \inner{\nabla_\si\si,\rmn}
\\
	+ \cos\theta\inner{v,\dot{\ga}}\inner{\nabla_{\dot{\ga}}v,\hat{\nu}}
	- \cos\theta\inner{v,\hat{\nu}}\inner{\nabla_{\dot{\ga}}v,\dot{\ga}}
\,\ed\tau}.
\end{multlined}
\end{align*}
We focus then on the terms on the boundary.
For convenience, let us define
\begin{equation}
X:=
\dvg_\Si\si\inner{\si,\rmn}
- \inner{s,H}\inner{\si,\rmn}
- 2\inner{\nabla_\si s,\rmn}
- \inner{\nabla_\si\si,\rmn}
+ \cos\theta\inner{v,\dot{\ga}}\inner{\nabla_{\dot{\ga}}v,\hat{\nu}}
- \cos\theta\inner{v,\hat{\nu}}\inner{\nabla_{\dot{\ga}}v,\dot{\ga}},
\label{hess-combined-x}
\end{equation}
which we shall use later.

Since
$\hat{\nu} = \frac{1}{\cos\theta}\rmn - \tan\theta N$,
we have that
\begin{align}
\cos\theta\inner{v,\dot{\ga}}\inner{\nabla_{\dot{\ga}}v,\hat{\nu}}
&=
\inner{v,\dot{\ga}}\inner{\nabla_{\dot{\ga}}v,\rmn}
- \sin\theta\inner{v,\dot{\ga}}\inner{\nabla_{\dot{\ga}}v,N}
\nonumber\\
&=
\inner{v,\dot{\ga}}\inner{\nabla_{\dot{\ga}}\si,\rmn}
+ \inner{v,\dot{\ga}}\inner{\nabla_{\dot{\ga}}s,\rmn}
- \sin\theta\inner{v,\dot{\ga}}\inner{A^{\pd M}\rbrak{\dot{\ga},v},N}
\nonumber\\
&=
\inner{\nabla_\si\si,\rmn}
- \inner{\si,\rmn}\inner{\nabla_\rmn\si,\rmn}
- \inner{\si,\dot{\ga}}\inner{A(\dot{\ga},\rmn),s}
- \sin\theta\inner{v,\dot{\ga}}\inner{A^{\pd M}\rbrak{\dot{\ga},v},N}
\nonumber
\end{align}
Secondly, we have
\begin{align}
-\cos\theta\inner{v,\hat{\nu}}\inner{\nabla_{\dot{\ga}}v,\dot{\ga}}
&=
-\inner{\nabla_{\dot{\ga}}v,\dot{\ga}}\inner{v,\rmn}
\nonumber\\
&=
- \inner{\nabla_{\dot{\ga}}\si,\dot{\ga}}\inner{\si,\rmn}
- \inner{\nabla_{\dot{\ga}}s,\dot{\ga}}\inner{\si,\rmn}
\nonumber\\
&=
- \inner{\si,\rmn}\inner{\nabla_{\dot{\ga}}\si,\dot{\ga}}
+ \inner{A(\dot{\ga},\dot{\ga}),\nu}\inner{s,\nu}\inner{\si,\rmn}.
\nonumber
\end{align}
Combining these two computations gives
\begin{multline*}
\cos\theta\inner{v,\dot{\ga}}\inner{\nabla_{\dot{\ga}},\hat{\nu}}
-\cos\theta\inner{v,\hat{\nu}}\inner{\nabla_{\dot{\ga}}v,\dot{\ga}}
=
\inner{\nabla_\si\si,\rmn}
- \inner{\si,\rmn}\dvg_\Si\si
+ \inner{A(\dot{\ga},\dot{\ga}),\nu}\inner{\si,\rmn}\inner{s,\nu}
\\
- \inner{A(\dot{\ga},\rmn),\nu}\inner{\si,\dot{\ga}}\inner{s,\nu}
- \sin\theta\inner{v,\dot{\ga}}\inner{A^{\pd M}\rbrak{\dot{\ga},v},N}.
\end{multline*}
But
\begin{align}
-2\inner{\nabla_\si s,\rmn}
&=
2\inner{A(\si,\rmn),s}
\nonumber\\
&=
2\inner{A(\dot{\ga},\rmn),\nu}\inner{\si,\dot{\ga}}\inner{s,\nu}
+ 2\inner{A(\rmn,\rmn),\nu}\inner{\si,\rmn}\inner{s,\nu},
\nonumber
\end{align}
and so combining this with \eqref{hess-combined-x} to compute $X$ yeilds many cancellations, giving
\begin{align}
X
&=
\inner{A(\si,\rmn),s}
- \sin\theta\inner{v,\dot{\ga}}\inner{A^{\pd M}\rbrak{\dot{\ga},v},N}.
\nonumber\\
&=
\inner{A(\rmn,\rmn),\nu}\inner{\si,\rmn}\inner{s,\nu}
+ \inner{A(\dot{\ga},\rmn),\nu}\inner{\si,\dot{\ga}}\inner{s,\nu}
- \sin\theta\inner{v,\dot{\ga}}\inner{A^{\pd M}\rbrak{\dot{\ga},v},N}
\nonumber
\intertext{and since $\inner{\si,\rmn}=\inner{v,\rmn}=\cot\theta\inner{v,\nu}$,}
&=
\cot\theta\inner{A(\rmn,\rmn),\nu}\inner{s,\nu}^2
+ \inner{A(\dot{\ga},\rmn),\nu}\inner{\si,\dot{\ga}}\inner{s,\nu}
- \sin\theta\inner{v,\dot{\ga}}\inner{A^{\pd M}\rbrak{\dot{\ga},v},N}.
\label{hess-combined-x-a}
\end{align}
Let us focus on the second term.
Since $\nu=-\frac{1}{\cos\theta}N+\tan\theta\rmn$, we have
\begin{align}
\inner{A(\dot{\ga},\rmn),\nu}\inner{s,\nu}\inner{\si,\dot{\ga}}
&=
- \sec\theta\inner{\nabla_{\dot{\ga}}\rmn,N}\inner{s,\nu}\inner{\si,\dot{\ga}}
\nonumber
\intertext{and since also $\rmn=\cos\theta\hat{\nu}+\sin\theta N$,}
&=
- \inner{\nabla_{\dot{\ga}}\hat{\nu},N}\inner{s,\nu}\inner{\si,\dot{\ga}}
- \frac{\sin\theta}{\cos\theta}\inner{\nabla_{\dot{\ga}}N,N}\inner{s,\nu}\inner{\si,\dot{\ga}}
\nonumber\\&
=
- \inner{A^{\pd M}(\dot{\ga},\hat{\nu}),N}\inner{s,\nu}\inner{\si,\dot{\ga}}.
\label{hess-combined-x-a1}
\end{align}
Notice however that
\begin{align}
0 &= \inner{\nabla_vv,N}
\nonumber\\&
=
\inner{A^{\pd M}(\dot{\ga},\dot{\ga}),N}\inner{v,\dot{\ga}}^2
+ \inner{A^{\pd M}(\hat{\nu},\hat{\nu}),N}\frac{\abs{s}^2}{\sin^2\theta}
+ 2\inner{A^{\pd M}(\dot{\ga},\hat{\nu}),N}\inner{v,\dot{\ga}}\inner{v,\hat{\nu}}
\nonumber\\&
=
\inner{A^{\pd M}(\dot{\ga},v),N}\inner{v,\dot{\ga}}
+ \inner{A^{\pd M}(\hat{\nu},\hat{\nu}),N}\frac{\abs{s}^2}{\sin^2\theta}
+ \inner{A^{\pd M}(\dot{\ga},\hat{\nu}),N}\inner{v,\dot{\ga}}\frac{\inner{v,\nu}}{\sin\theta}.
\label{hess-combined-x-a2}
\end{align}
Combining \eqref{hess-combined-x-a}, \eqref{hess-combined-x-a1}, and \eqref{hess-combined-x-a2} gives
$$
X
=
\cot\theta\inner{A(\rmn,\rmn),\nu}\abs{s}^2
+ \inner{A^{\pd M}(\hat{\nu},\hat{\nu}),N}\frac{\abs{s}^2}{\sin\theta},
$$
as desired.
\end{proof}

\begin{lemma}
\label{lem:hess-mod-energy}
Let $u:\Si\to M$ be a boundaried branched conformal immersion.
Then for any $v\in\Ga_{\pd M}(u^*TM)$, we have
\begin{multline}
\secondvar{\calE^{h,\theta}}{u}{v,v}
= \geomint{\Si}{
	\abs{\nabla v}^2
	- \Rm^M(v,E_i,E_i,v)
\\
	+ h\rbrak{
		\Vol^M(v,\nabla_{e^{-\la}u_x}v,e^{-\la}u_y)
		+ \Vol^M(v,e^{-\la}u_x,\nabla_{e^{-\la}u_y}v)
	}
	\,\ed\Si
}
\\
- \geomint{\pd\Si}{
\sin\theta\inner{A^{\pd M}(\dot{\ga},\dot{\ga}),N}\abs{v}^2
\,\ed\tau.
}
\end{multline}
\end{lemma}

\printbibliography

\bigskip
\noindent
\url{https://lseem.xyz}\\
\textsc{School of Mathematics, University of Leeds\\
Leeds LS2 9JT, United Kingdom.}

\end{document}